\documentclass{amsart}
\usepackage{dsfont}
\usepackage{helvet}
\usepackage{color}
\usepackage[bbgreekl]{mathbbol}
\usepackage{mathrsfs}  
\usepackage{bbm}

\usepackage{amsmath,amsxtra,amsthm,amssymb,xr}
\usepackage[all]{xy}
\newtheorem{theorem}{Theorem}[section]
\newtheorem{lemma}[theorem]{Lemma}
\newtheorem{conj}[theorem]{Conjecture}
\newtheorem{proposition}[theorem]{Proposition}
\newtheorem{corollary}[theorem]{Corollary}
\newtheorem{defn}[theorem]{Definition}

\newtheorem{remark}[theorem]{Remark}

\newcommand{\ooo}{\mathfrak{O}}

\newcommand{\Gal}{\operatorname{Gal}}

\newcommand{\sign}{\operatorname{sign}}
\renewcommand{\d}{\operatorname{d}}
\newcommand{\DD}{\mathbb{D}}

\newcommand{\NN}{\mathbb{N}}
\newcommand{\QQ}{\mathbb{Q}}
\newcommand{\Qp}{\mathbb{Q}_p}
\newcommand{\Zp}{\mathbb{Z}_p}
\newcommand{\ZZ}{\mathbb{Z}}
\newcommand{\fa}{\mathfrak{a}}
\newcommand{\fl}{\mathfrak{l}}
\newcommand{\fm}{\mathfrak{m}}
\newcommand{\fn}{\mathfrak{n}}

\newcommand{\FFF}{\mathcal{F}}
\newcommand{\g}{\mathbf{g}}

\newcommand{\LSU}{\mathcal{L}^{\rm SU}}

\newcommand{\ord}{\mathrm{ord}}

\newcommand{\fp}{\mathfrak{p}}
\newcommand{\fq}{\mathfrak{q}}

\newcommand{\calL}{\mathcal{L}}

\newcommand{\calO}{\mathcal{O}}
\newcommand{\cE}{\mathcal{E}}
\newcommand{\Iw}{\mathrm{Iw}}

\newcommand{\GL}{\mathrm{GL}}

\newcommand{\col}{\mathrm{Col}}

\newcommand{\cyc}{\textup{cyc}}

\newcommand{\al}{\mathfrak{L}}

\newcommand{\ff}{\mathfrak{f}}

\newcommand{\Hom}{\mathrm{Hom}}
\newcommand{\Char}{\mathrm{char}}

\newcommand{\ac}{\textup{ac}}
\newcommand{\LL}{\Lambda}
\newcommand{\TT}{\mathbb{T}}
\newcommand{\oo}{\mathfrak{O}}
\newcommand{\RR}{\mathcal{R}}
\newcommand{\f}{\textup{f}}
\newcommand{\plusf}{+/\textup{f}}
\newcommand{\Gr}{\textup{Gr}}
\newcommand{\FFgr}{\FFF_\textup{Gr}}
\newcommand{\lra}{\longrightarrow}
\newcommand{\ra}{\rightarrow}
\newcommand{\res}{\textup{res}}

\newcommand{\Bj}{\mathbf{j}}
\newcommand{\Ba}{\mathbf{a}}
\newcommand{\Bf}{\mathbf{f}}
\newcommand{\Bg}{\mathbf{g}}
\newcommand{\ur}{\textup{ur}}

\newcommand{\BF}{\textup{BF}}
\newcommand{\cP}{\mathcal{P}}

\begin{document}

\title[Anticyclotomic $p$-ordinary Iwasawa Theory]{Anticyclotomic  $p$-ordinary Iwasawa Theory of Elliptic Modular Forms}

\begin{abstract}
This is the first in a series of articles where we will study the Iwasawa theory of an elliptic modular form $f$ along the anticyclotomic $\ZZ_p$-tower of an imaginary quadratic field $K$ where the prime $p$ splits completely. Our goal in this portion is to prove the Iwasawa main conjecture for suitable twists of $f$ assuming that $f$ is $p$-ordinary, both in the \emph{definite} and \emph{indefinite} setups simultaneously, via an analysis of Beilinson-Flach elements. 
\end{abstract}

\author{K\^az\i m B\"uy\"ukboduk}
\address{K\^az\i m B\"uy\"ukboduk\newline
Ko\c{c} University, Mathematics  \\
Rumeli Feneri Yolu, 34450 Sariyer \\ 
Istanbul, Turkey}
\email{kbuyukboduk@ku.edu.tr}

\author{Antonio Lei}
\address{Antonio Lei\newline
D\'epartement de Math\'ematiques et de Statistique\\
Universit\'e Laval, Pavillion Alexandre-Vachon\\
1045 Avenue de la M\'edecine\\
Qu\'ebec, QC\\
Canada G1V 0A6}
\email{antonio.lei@mat.ulaval.ca}

\thanks{The first named author is partially supported by the Turkish Academy of Sciences and T\"UB\.ITAK Grant 113F059. The second named author is supported by the NSERC Discovery Grants Program 05710.}
\maketitle
\tableofcontents
\section{Introduction}

Fix forever a rational prime $p \geq 5$ and an imaginary quadratic number field $K$ in which $(p)=\fp\fp^c$ splits. Let $\eta$ denote the quadratic Dirichlet character of $\QQ$ associated to $K/\QQ$. We also fix once and for all an embedding $\iota_p: \overline{\QQ}\hookrightarrow \mathbb{C}_p$ and suppose that the prime $\fp$ of $K$ is the prime induced from this embedding. Let $K_\infty$ denote the maximal $\ZZ_p$-power extension of $K$, so that we have $\Gamma:=\textup{Gal}(K_\infty/K)\cong \ZZ_p^2$. We let $D_\infty/K$ denote the anticyclotomic $\ZZ_p$-extension of $K$ (so that $\Delta_{\ac}:=\Gal(D_\infty/\QQ)$ is an infinite dihedral group) contained in $K_\infty$ and $K_\cyc/K$ the cyclotomic $\ZZ_p$-extension of $K$ (so that $\Delta_\cyc:=\Gal(K_\cyc/\QQ)$ is abelian). We write $\Gamma^{\ac}:=\Gal(D_\infty/K)$ and $\Gamma^\cyc:=\Gal(K_\cyc/K)$. Set $\LL=\ZZ_p[[\Gamma]]$, $\LL_{\ac}=\ZZ_p[[\Gamma^\ac]]$ and $\LL_{\cyc}=\ZZ_p[[\Gamma^\cyc]]$.

Let $f=\sum a_nq^n \in S_k(\Gamma_0(N))$ be a normalized elliptic newform of (even) weight $k$ and level $\Gamma_0(N)$. When the weight of $f$ equals $2$, assume in addition that $p>5$ and $p\nmid N$. Let $L(f/F,s)$ denote the Hecke $L$-function associated to the base change of $f$ to a number field $F$. Suppose that all $\iota_p(a_n) \in L$, a finite extension of $\QQ_p$ and let $W_f$ denote Deligne's two-dimensional representation attached to $f$, with coefficients in $L$. We shall write $V_f:=W_f(k/2)$ for the self-dual twist of $V_f$ and occasionally refer to it as the central critical twist of Deligne's representation. Denote the valuation ring of $L$ by $\mathfrak{o}=\mathfrak{o}_L$ and fix a Galois-stable $\mathfrak{o}$-lattice $U_f$ in $W_f$. Set $T_f=U_f(k/2)$.


Suppose that $E_{/\QQ}$ is an elliptic curve with complex multiplication by $K$. When the sign $\epsilon(E/K)$ of the functional equation of $L(E/K,s)$ is $+1$ (which we henceforth refer as the definite case) and $p$ is a good ordinary prime, Iwasawa's main conjecture for $E$ along $D_\infty$ follows from the two-variable main conjecture proved by Rubin in \cite{rubinmainconj}.  When the sign $\epsilon(E/K)$ equals $-1$ (which we henceforth refer as the indefinite case), Agboola and Howard in \cite{agboolahowardordinary} proved the anticyclotomic main conjecture and this was generalized by Arnold in~\cite{arnoldhigherweightanticyclo} to higher weight CM forms. When $p$ is a supersingular prime, Agboola and Howard in~\cite{agboolahowardsupersingular} proved also a weak form of the main conjecture for the elliptic curve $E$. This was extended to cover any CM form of arbitrary weight by the first-named author in~\cite{kbbanticyclossCM}. The key ingredients in all this work are the elliptic unit Euler system and the reciprocity laws of Coates-Wiles~\cite{coateswiles77, wiles78} and Kato~\cite{katoreciprocity99}, as well as Rubin's two-variable main conjecture. 

As part of this work, we shall carry a similar task out when the modular form in question is $p$-ordinary and non-CM. In this situation, in place of the elliptic units, we shall make use of the Beilinson-Flach elements\footnote{In Section~\ref{subsubsec:ESoverK} below, we explain how to obtain an Euler system of Beilinson-Flach classes over $K$ out of those classes constructed in \cite{KLZ2} for Ranking-Selberg convolutions. The relation of this portion to \cite{KLZ2} should be viewed as the counterpart of \cite{LLZ2} to \cite{LLZ1}. In particular, the contents of Section~\ref{subsubsec:ESoverK} brings the higher weight set up to equal footing with the weight-$2$ situation.}$^,$ \footnote{During the preparation of this article, F. Castella released a proof along these lines for eigenforms of weight $2$, concentrating on the indefinite case (and still when the prime $p$ splits and $f$ is $p$-ordinary).} of~\cite{KLZ2} as well as the explicit reciprocity laws of~\cite{KLZ2}. Note that results in this vain has been previously obtained also by Bertolini and Darmon~\cite{BD05} in the definite setting (and when the weight of the form in question equals $2$) making use of Heegner points (which are not directly available in the definite setting, but they are instead constructed using level-raising arguments in loc. cit.); these results have been extended by Chida and Hsieh in \cite{chidahsiehanticyclomainconjformodformscomposito} to higher weight forms (also recently by Castella-Kim-Longo in \cite{castellakimlongo} using a different method which is  based on an anticyclotomic  version of the variational results of \cite{EPW}, which allows the authors to reduce to weight $2$). In the indefinite case, Howard~\cite{howardcompositio1,how2LamdaadicGZ} has proven a form of our Theorem~\ref{thmmainpordinaryintro}(ii) below also for forms of weight $2$ making developing the Heegner point Kolyvagin system machine; however, contrary to \cite{BD05}, it is necessary to assume in \cite{how2LamdaadicGZ} that the prime $p$ splits in $K/\QQ$. While \cite{BD05,howardcompositio1} allows the treatment of the case when the prime $p$ is inert in $K/\QQ$ (this case escapes the methods of the current article), our approach here has the advantage of offering a simultaneous treatment of the definite and indefinite set up, as well as an access to a study of modular forms of higher weights. Furthermore, our approach naturally extends (as studied in our forthcoming note \cite{kbbleipnonord}) to treat the $p$-non-ordinary modular forms, again both in the definite and indefinite situations simultaneously under the hypothesis that  $p$ splits in $K/\QQ$. We remark that Darmon and Iovita in \cite{darmoniovita} has extended the methods of \cite{BD05} to prove (also in the definite set up) a divisibility in a plus/minus anticyclotomic main conjecture when the weight of the form in question equals $2$ (and contrary to \cite{BD05} assuming that the prime $p$ is split in $K/\QQ$, as we do here).

Before we state our main results, we introduce a technical condition which is present at some portions of our work. Let $\mathbf{f}$ be the Hida family\footnote{In this article, a Hida family means the localization of the Hecke algebra (of an appropriate level) at a maximal ideal, following the convention of \cite[\S7.2]{KLZ2}.   } carrying the $p$-ordinary stabilization of $f$ as its weight-$k$ specialization. Let $\Psi$ denote the $U_p$-eigenvalue  of the family $\mathbf{f}$ and write $\lambda$ for its specialization at $k$ (so that the $\lambda$ is the unit root of the Hecke polynomial of $f$ at $p$).  Consider the following condition on $\mathbf{f}$:
\\\\
\textbf{(wt-2)} The weight-$2$ specialization of $\Psi$ is not equal to $\pm1$ and that $k \equiv 2$ $\mod p-1$.


We remark that the portion of the hypothesis \textbf{(wt-2)} concerning the $U_p$- eigenvalue ensures that the weight-$2$ (with trivial wild character) specialization of $\mathbf{f}$ is $p$-old. The second portion of \textbf{(wt-2)} guarantees that the weight-$2$ specialization (with trivial wild character) has trivial nebentype. We will work under \textbf{(wt-2)}  only when we rely on the work of Wan in the \emph{indefinite case} (which we rely on in order to verify that the bounds we obtain here through an Euler system argument are indeed sharp).

We are now ready to state our main results. 

\begin{theorem}
\label{thmmainpordinaryintro}
Assume that the hypotheses \textup{\textbf{(H.Im.)}}, \textup{\textbf{(H.Dist.)}}, \textup{\textbf{(H.SS.)}} and \textup{\textbf{(H.nEZ.)}} (which we introduce below) hold true. Factor $N=p^{\alpha_p}N^+N^-$ where $N^+$ (resp., $N^-$) is only divisible by primes that are split (resp., inert) in $K$. Suppose that $N^-$ is square-free and $\alpha$ is any ring class character modulo $\ff p^\infty$ (where the choice of the modulus $\ff$ is made precise below) for which we have $\alpha(\fp)\neq \alpha(\fp^c)$.
\\\\
\textup{\textbf{(i)}} Assume that $N^-$ is a product of odd number of primes, $p^2$ does not divide $N$ and $({\rm{disc}(K/\QQ)},{N})=1$. Then the characteristic ideal of a suitably defined Greenberg Selmer group $\mathfrak{X}(f\otimes\alpha/D_\infty)$  contains the anticyclotomic projection $\mathfrak{L}_{f,0}^{(\alpha)}$ of the Hida-Perrin-Riou $p$-adic $L$-function, after inverting $p$. If in addition we assume that $p \nmid N$, then $\Char\left(\mathfrak{X}(f\otimes\alpha/D_\infty)\right)\otimes L$ is generated by $\mathfrak{L}_{f,0}^{(\alpha)}$.
\\\\
\textup{\textbf{(ii)}} Assume that $N^-$ is a square-free product of even number of primes and $p$ is prime to $N\rm{disc}(K/\QQ)$. Then $\al_{f,0}^{(\alpha)}=0$ and the Greenberg Selmer group $\mathfrak{X}(f\otimes\alpha/D_\infty)$ has rank one. Let $\al_{f,1}^{(\alpha)}$ denote the derivative of the Hida-Perrin-Riou $p$-adic $L$-function in the cyclotomic direction, restricted to the anticyclotomic line. Then,
$$\al_{f,1}^{(\alpha)}\in \mathfrak{Reg}_{\ac}\cdot\Char\left(\mathfrak{X}(f\otimes\alpha/D_\infty)_{\textup{tor}}\right)\otimes L$$ 
where $\mathfrak{Reg}_{\ac}$ is a regulator defined in terms of a $\LL_\ac$-adic height pairing. 

If we assume in addition that $N$ is square free, that the hypothesis \textup{\textbf{(wt-2)}} holds true and that there exists a prime $q \mid N^-$ such that $\overline{\rho}_f$ is ramified at $q$, then the ideal $\mathfrak{Reg}_{\ac}\cdot\Char\left(\mathfrak{X}(f\otimes\alpha/D_\infty)_{\textup{tor}}\right)$ is in fact generated by $\al_{f,1}^{(\alpha)}\,.$
\end{theorem}
\begin{remark}\label{rk:redundant}
As explained in \cite[Remark 7.2.7]{KLZ2}, the hypothesis \textup{\textbf{(H.Dist.)}} (that we introduce below) is redundant in the presence of the hypothesis \textup{\textbf{(wt-2)}}.
\end{remark}
\begin{remark} The first part of this theorem is a combination of Theorem~\ref{thm:mainconjwithoutpadicLfunc}(i), the first reciprocity law (\ref{eqn:firstreciprocitylaw}) for the Beilinson-Flach elements and Theorem~\ref{thm:skinnerurbanwan} below. We note that the result we record in the second portion of (i) is an immediate consequence\footnote{Except for the comparison of the Hida-Perrin-Riou $p$-adic $L$-function to the Skinner-Urban $p$-adic $L$-function that we explain in the proof of Theorem~\ref{thm:skinnerurbanwan} (which we deduce from our computations in the proof of Theorem~\ref{thm:sigma1interpolationformula}) below. This comparison has been suggested in \cite{skinnerurbanmainconj} and has been used by several authors since then; however we were unable to locate a valid proof in the literature. As far as we are aware, our article is the first instance where this comparison is worked out in detail.} of the works of Kato, Skinner-Urban and Wan concerning the Iwasawa main conjectures in the definite case. Our contribution in (i) therefore is only in the first half of this part: Our hypotheses on $N$ in the first portion of  Theorem~\ref{thmmainpordinaryintro}(i) are somewhat weaker then those assumed in \cite{chidahsiehanticyclomainconjformodformscomposito} (which generalizes \cite{BD05}) in that we do not require $p$ be prime to $N$. 

Theorem~\ref{thmmainpordinaryintro}(ii) is Theorem~\ref{thm:torsionpartGreenberg} in the main body of our article. So far as the authors are aware, the entirety of the results announced as part of this portion are new (in the level of generality they are presented). For eigenforms of weight $2$, Howard in \cite{how2LamdaadicGZ} has obtained results similar to Theorem~\ref{thmmainpordinaryintro}(ii). Our methods to prove Theorem~\ref{thmmainpordinaryintro}(ii) are altogether different compared to his: Here we utilize Beilinson-Flach elements along the anticyclotomic tower, and the crucial point is to determine their $p$-local positions, which we do so relying on the explicit reciprocity laws they verify. We further remark that in the preprint \cite{castellaheegner} (that was released while the initial version of the current work was being drafted), Castella has independently employed a similar approach to study the anticyclotomic Iwasawa theory of modular forms of weight $2$ in the indefinite case (and recover the contents of Theorem~\ref{thmmainpordinaryintro}(ii) for forms of weight $2$).
\end{remark}

\begin{remark}
In our preprint~\cite{kbbleipnonord},  we expand on the methods here and treat the case when $f$ is $p$-non-ordinary. 
\end{remark}

There has been tremendous activity related to the contents of Theorem~\ref{thmmainpordinaryintro}; we record here some of those which are instrumental in its proof. Skinner and Urban in \cite{skinnerurbanmainconj} proved a two variable Iwasawa main conjecture for $p$-ordinary forms in the definite setting, and Wan in \cite{xinwanwanrankinselberg} extended this result to cover the indefinite situation as well. These two results rely on the so-called Eisenstein-Klingen ideal method. This method initially produces a lower bound for the Selmer group in question. On the other hand, making use of the recent work of Kings-Loeffler-Zerbes \cite{KLZ2} on Beilinson-Flach elements (which is adjusted slightly in Section~\ref{subsubsec:ESoverK} below for our purposes here), we may produce an upper bound for the same Selmer groups. This yields the equality in the statement of the two-variable Iwasawa's main conjecture. In the definite setting, Chida and Hsieh~\cite{chidahsiehanticyclomainconjformodformscomposito} proved analogous result for higher weight modular forms, extending the earlier work of Bertolini and Darmon \cite{BD05} but using very different techniques then ours. We finally remark that Castella-Kim-Longo in \cite{castellakimlongo} also proved results in line with Theorem~\ref{thmmainpordinaryintro}(i) using a yet different strategy. Their method is  based on an anticyclotomic  version of the variational results of \cite{EPW} and allows the authors to reduce to weight $2$.


\subsection*{Acknowledgements} We would like to thank Henri Darmon and Chan-Ho Kim for enlightening discussions during the preparation of this paper.

\subsection*{Notation and hypothesis}
As above, let $f$ be a cuspidal eigenform of arbitrary even weight $k\geq 2$ which is not a CM-form, level $N$ and with trivial nebentypus. Let $\rho_f:G_{\QQ}\rightarrow \textup{GL}(W_f)$ denote Deligne's Galois representation attached to $f$. Fix also forever a modulus $\mathfrak{f}$ of $K$ that is prime to $p$. We will assume that all the following three hypotheses hold true:
\\\\
\textup{\textbf{(H.Im.)}} The image of $\rho_f$ contains a conjugate of $\textup{SL}_2(\ZZ_p)$.
\\\\
\textup{\textbf{(H.Dist.)}} $f$ is $p$-distinguished (in the sense of \cite[Definition~7.2.6]{KLZ2}).
\\\\
\textup{\textbf{(H.SS.)}} The order of the ray class group of $K$ modulo $\mathfrak{f}$ is prime to $p$.
\\\\
Fix a ring class character $\alpha$ modulo $\mathfrak{f}p^\infty$ of finite order, for which we have $\alpha(\mathfrak{p})\neq \alpha(\mathfrak{p}^c)$. Let $\omega$ denote the $p$-adic Teichm\"uller character. In order  the locally restricted Euler system machinery (as we have recorded in \cite[Appendix A]{kbbleiPLMS}; see also \cite[\S 12]{KLZ2}) to apply in Section~\ref{subsec:ESargument}, we will also need the following hypothesis:
\\\\
\textup{\textbf{(H.nEZ.)}} $v_p\left(\alpha^{-1}\omega^{1-k/2}(\fp^c)\cdot\lambda-1\right)=0$.

We remark that the hypothesis \textup{\textbf{(H.nEZ.)}} is redundant unless  $k\equiv 2$ mod $2(p-1)$. 

In addition to the four properties listed above, we will also consider the following conditions on the sign of the functional equation (at the central critical point): 

(Sign $-$)\,\,\,\,\,\,\,\,\,\,\,\,\,\,\,\,\,\,\,\,\,\,\,\,\,\,\,\,\,\,\,\,\,\,\,\, \,\,\,\,\,\,\,\,\,\,\,\,\,\,\,\,$\epsilon(f/K)=-1$\,.

(Sign $+$)\,\,\,\,\,\,\,\,\,\,\,\,\,\,\,\,\,\,\,\,\,\,\,\,\,\,\,\,\,\,\,\,\,\,\,\, \,\,\,\,\,\,\,\,\,\,\,\,\,\,\,\,$\epsilon(f/K)=+1$\,.
\\
Note that (Sign $-$) ensures that the Hecke $L$-function $L(f/K,s)$ vanishes to odd order at the central critical point $s=k/2$ and places us in the indefinite case in the terminology we have fixed above. We note that the hypotheses of Theorem~\ref{thmmainpordinaryintro}(i) on $N$ (respectively, those in Theorem~\ref{thmmainpordinaryintro}(ii)) ensures that (Sign $+$) (respectively, the condition (Sign $-$)) holds true.

Let $\Phi$ denote the completion of the maximal unramified extension of $L$ and let $\ooo$ denote its valuation ring. For any complete extension $E$ of $\QQ_p$ with valuation ring $\mathcal{O}$, we shall write $\LL^\mathcal{O}$ in place of the ring $\LL\widehat{\otimes}\mathcal{O}$ and likewise define $\LL_{\ac}^\mathcal{O}$ and $\LL_{\cyc}^\mathcal{O}$. We set $\mathcal{R}_E=\LL^\mathcal{O}\otimes_{\mathcal{O}}E$ and similarly define the rings $\mathcal{R}_E^\ac$ and $\mathcal{R}_E^\cyc$. We write $\varpi_E\in \mathcal{O}$ for a fixed uniformizing element of $E$ and when $E=L$, we simply write $\varpi$ instead of $\varpi_{L}$.

For  $?=\emptyset,\cyc$ or $\ac$, let $\iota: \LL^\mathcal{O}_{?}\ra \LL^\mathcal{O}_{?}$ denote the involution induced by the map $\gamma\mapsto\gamma^{-1}$ on the group like elements of $\LL^\mathcal{O}_?$. For a given $\LL^\mathcal{O}_?$-module $M$, we write $M^\iota$ for the module $M\otimes_{\LL_{?}^\mathcal{O},\iota}\LL^\mathcal{O}_{?}$. For $?=\emptyset,\cyc$ or $\ac$, the Iwasawa algebra $ \LL_{?}$ is equipped with a natural Galois action, given by the canonical character 
$$\kappa: G_{K}\twoheadrightarrow\Gamma^{?}\hookrightarrow \LL_?^\times\,.$$
We define $\LL^{\mathcal{O},\iota}_?$ as the rank one $\LL^\mathcal{O}_?$-module twisted as above.

Following \cite[\S6.1]{LLZ2}, we shall consider the following sets of Hecke characters of $\mathbb{A}_K^\times$:
\begin{itemize}
\item $\Sigma^{(1)}:=\{\hbox{Characters of infinity type } (r, s) \,:\, 1-k/2\leq r,s \leq k/2-1\}$
\item $\Sigma^{(2)}:=\{\hbox{Characters of infinity type } (r, s) \,:\, r\leq -k/2, s \geq k/2\}$
\item $\Sigma^{(2^\prime)}:=\{\hbox{Characters }\hbox{ of infinity type } (r, s) \,:\, s\leq -k/2, r \geq k/2\}$\,.
\end{itemize}
The union $\Sigma=\Sigma^{(1)}\sqcup \Sigma^{(2)}\sqcup\Sigma^{(2^\prime)}$ is called the \emph{critical Hecke characters} for $f$.
Our convention here is such that a Hecke character $\xi$ of infinity type $(\ell_1,\ell_2)$ belongs to $\Sigma^{(i)}$ in the sense of \cite[Definition 4.1]{bertolinidarmonprasanna13} if and only if $\chi_\cyc^{-k/2}\xi^{-1}$ belongs to our $\Sigma^{(i)}$  (for $i=1,2,2^\prime$). This alteration is due to the fact that Bertolini-Darmon-Prasanna collect together all characters for which the value $L(f\otimes\xi^{-1},0)$ is critical in the sense of Deligne, whereas we consider those for which the value $L(f\otimes\xi,k/2)$ is critical. 

We also set $\Sigma^{(i)}_{\textup{cc}}:=\{\xi \in \Sigma^{(i)} \hbox{ of infinity type } (r,s): r+s=0 \}$  (for $i=1,2,2^\prime$) and call them the central critical characters. Note that all anticyclotomic characters (which we regard as a Hecke character by class field theory via the geometrically normalized Artin map) which are critical are in fact central critical.

Throughout this article, the functor $(\,\,)\mapsto (\,\,)^*$ will stand for passing to linear duals, whereas  $(\,\,)\mapsto (\,\,)^{\vee}$ passing to Pontryagin duals.
\section{The analytic theory}
\label{sec:analyticLfunctions}
\subsection{$L$-functions for Rankin-Selberg convolutions}
Let $f$ be a cuspidal eigenform of level $\Gamma_0(N)$ (in particular, of trivial central character) and suppose that $\xi$ is an algebraic Hecke character of $K$ of infinity type $(a,b)$.  We will be interested in Rankin-Selberg convolutions $f\otimes \psi$. More precisely, let $\pi=\pi_f$ denote the automorphic representation corresponding to $f$ and let $\pi(\xi)$ be that corresponding to $\xi$. The Rankin-Selberg (automorphic) $L$-function\footnote{We will intentionally write $L(s,\star)$ for automorphic $L$-functions, whereas we shall denote the Hecke-$L$-functions by $L(\star,s)$.} $L(s,\pi\times\pi(\xi))$ is defined as an Euler product (where each Euler factor will be of degree $4$) and in fact, it is the same as the the twisted base-change $L$-function $L(s,\pi_K,\xi)$ up to a shift, where $\pi_K$ is the base change of $\pi$ to $\GL_2(\mathbb{A}_K)$. 

The Rankin-Selberg $L$-function $L(s,\pi\times\pi(\xi))$ admits a meromorphic continuation (where the only possibile poles are simple poles at $s=0$ and $1$) to the entire complex plane by the work of Jacquet and Jacquet-Langlands.    Furthermore, it verifies a functional equation of the form 
$$\LL(s,\pi\times\pi(\xi))=\epsilon(s,\pi\times\pi(\xi))\,\LL(1-s,\widetilde{\pi}\times\pi(\xi^{-1}))$$ 
where $\widetilde{\pi}$ is the contragredient of $\pi$ and $\LL(s,\star)$ stands for the completed (automorphic) $L$-function. The completed $L$-function is in fact entire, unless $\pi(\xi)\cong \widetilde{\pi}\otimes |\cdot|^r$ for some real number $r$. We shall be working with a non-CM form in the main body of our article, and therefore the $L$-functions we will deal with here will be always holomorphic everywhere.

The only case of interest for us will be the situation when the algebraic Hecke character $\xi$ is anticyclotomic (equivalently, denoting the character of $G_K$ associated to $\xi$ by class field theory also by $\xi$, we have $\xi(c\sigma c^{-1})=\xi^{-1}(\sigma)$ for (an arbitrary lift of) the generator $c$ of  $\Gal(K/\QQ)$ and every $\sigma \in G_K$), say of infinity type $(m,-m)$. This means that we have $\xi\big{|}_{\mathbb{A}_\QQ^\times}=\mathds{1}$  and that $\LL(s,\pi\times\pi(\xi))=\LL(s,\widetilde{\pi}\times\pi(\xi^{-1}))$. The functional equation now takes the following form 
$$\LL(s,\pi\times\pi(\xi))=\epsilon(s,\pi\times\pi(\xi))\,\LL(1-s,{\pi}\times\pi(\xi))\,$$
with sign $\epsilon(f\otimes\xi/K):=\epsilon(1/2,\pi\times\pi(\xi)) \in \{\pm1\}\,.$

We note that the Hecke $L$-function $L(f\otimes\xi/K,s)$ is related to $L(s,\pi\times\pi(\xi))$ via
$$L\left(f\otimes\xi/K,s+\frac{k-1}{2}\right)=L(s,\pi\times\pi(\xi))$$ 
when $\xi$ is an anticyclotomic Hecke character.

\subsection{$p$-adic $L$-function of Hida and Perrin-Riou}\label{S:padicL}

Let $\Bf_1$ and $\Bf_2$ be two Hida families of tame levels $N_1$ and $N_2$ respectively. Suppose that $N$ is an integer divisible by both $N_1$ and $N_2$ but not by $p$. There exists a 3-variable $p$-adic $L$-function $L_p(\Bf_1,\Bf_2,s)$, where $s$ is the cyclotomic variable. Here, we use the normalization in \cite[\S5.3]{LLZ1}. More precisely, if $f_1$ and $f_2$ are weight $k$ (resp., weight $l$) specializations of $\Bf_1$ and $\Bf_2$, the value of the $p$-adic $L$-function is given by
\[
L_p(f_1,f_2,j)=\frac{\left\langle\overline{ f_1},e_{\rm ord}\left(f_2^{[p]}\times\cE_{1/N}(j-l,k-1-j)\right)\right\rangle_{N,p}}{\langle f_1,f_1\rangle_{N,p}},
\]
where $\overline{f_1}$ denotes the complex conjugate of $f_1$, $f_2^{[p]}$ is the $p$-depletion of $f_2$, $\langle\sim,\sim\rangle_{N,p}$ denotes the Petersson inner product  at level $\Gamma_1(N)\cap \Gamma_0(p)$ and $\cE_{\alpha}(\phi_1,\phi_2)$ denotes the $p$-depleted Eisenstein series 
\begin{align}
&\sum_{n\ge 1,p\nmid n}\left(\sum_{0<d|n}\phi_1(d)\phi_2(n/d)\left[e^{2\pi id/N}-\phi_1\phi_2(-1)e^{-2\pi id/N}\right]\right)q^n\notag\\
=&\sum_{n\ge 1,p\nmid n}\left(\sum_{d|n}\sign(d)\phi_1(d)\phi_2(n/d) e^{2\pi id/N}\right)q^n\label{eq:eisenstein}
\end{align}
whenever $\phi_1$ and $\phi_2$ are two characters on $\Zp^\times$ and $\alpha\in\frac{1}{N}\ZZ/\ZZ$, as given by \cite[Definition~5.3.1]{LLZ1}. Note that we write our characters additively here (so an integer $j$ that appears as an argument of a $p$-adic $L$-function stands for its evaluation under the character $\chi_\cyc^j$).  We recall from \cite[Theorem~2.7.3 and Remark~2.7.4(i)]{KLZ2} that this $p$-adic $L$-function has the following interpolation formula. If $f_1$ and $f_2$ are as above with levels coprime to $p$, let  $\alpha_i$ and $\beta_i$ be the roots to the Hecke polynomial of $f_i$ at $p$, with $\alpha_i$ being the unit root. If $j$ is an integer such that $l\le j\le k-1$ and $\chi$ is a finite character on $\Gamma^\cyc$, then
\begin{equation}\label{eq:interpolationHida}
L_p(f_1,f_2,j+\chi)=\frac{\cE(f_1,f_2,j+\chi)}{\cE(f_1)\cE^*(f_1)}\times\frac{i^{k-l}N^{2j-k-l+2}\Gamma(j)\Gamma(s-l+1)L(f,g,\chi^{-1},j)}{2^{2j+k-l}\pi^{2j+1-l}\langle f_1,f_1\rangle_{N_{f_1}}},
\end{equation}
where $\cE(f_1)=1-\beta_1/p\alpha_1$, $\cE^*(f_1)=1-\beta_1/\alpha_1$ and
\[
\cE(f_1,f_2,j+\chi)=\begin{cases}
\left(1-\frac{p^{j-1}}{\alpha_1\alpha_2}\right)\left(1-\frac{p^{j-1}}{\alpha_1\beta_2}\right)\left(1-\frac{\beta_1\alpha_2}{p^j}\right)\left(1-\frac{\beta_1\beta_2}{p^j}\right)&\text{if $\chi$ is trivial,}\\
\tau(\chi)^2\cdot \left(\frac{p^{2j-2}}{\alpha_1^2\alpha_2\beta_2}\right)^n&\text{if $\chi$ is of conductor $p^n>1$.}
\end{cases}
\]

Recall that $\ff$ is a modulus of $K$ with $(p,\ff)=1$. We write $H_{\ff p^\infty}$ for the ray class group of $K$ of conductor $\ff p^\infty$. As in \cite[\S6.2]{LLZ2}, we define
\[
\Theta=\sum_{\fa:(\fa,p)=1}[\fa]q^{N_{K/\QQ}(\fa)}\in\Lambda(H_{\ff p^\infty})[[ q]],
\]
where $\fa$ runs over ideals of $K$ and $[\fa]$ denotes the corresponding element of $H_{\ff p^\infty}$ via Artin reciprocity map. Given any character $\omega$ of $H_{\ff p^\infty}$, $\Theta(\omega)$ is then the $p$-depleted theta series attached to $\omega$. Note that its level divides $N_{K/\QQ}(\ff)\cdot{\rm disc}(K/\QQ)\cdot p^\infty$. On replacing $N$ by the lowest common multiple of the level of $f$ and $N_{K/\QQ}(\ff)\cdot{\rm disc}(K/\QQ)$ if necessary, we define a 2-variable $p$-adic $L$-function 
\[
L_p(f/K,\Sigma^{(1)})\in\Lambda_L(H_{\ff p^\infty}):=\mathfrak{o}[[H_{\ff p^\infty}]]\otimes L,
\]
which assigns a character $\omega$ of $H_{\ff p^\infty}$ the value
\[
L_p(f/K,\Sigma^{(1)})(\omega)=\frac{\left\langle \overline{ f^\lambda},e_\ord(\Theta(\omega)\times\cE_{1/N}(k/2-1-\omega_\QQ,k/2-1))\right\rangle_{N,p}}{\langle  {f^\lambda}, {f^\lambda} \rangle_{N,p}} ,
\]
where $ f^\lambda$ is the ordinary $p$-stabilization of $f$ and $\omega_\QQ$ is the character given by the composition of $\omega$ with $\Zp^\times\hookrightarrow (\calO_K\times\Zp)^
\times\rightarrow H_{\ff p^\infty}$. 
 We remark that if $f^\lambda$ and $\omega$ vary inside a Hida family $\Bf_1$ and $\Bf_2$ (the latter being a CM Hida family over $\Lambda(H_{\ff\fp^\infty})$), we recover the Hida $p$-adic $L$-function $L_p(\Bf_1,\Bf_2,s)$. More specifically, if  $g$ is the specialization of $\Bf_2$ that corresponds to the theta series attached to $\omega$,  we have the formula
 \[
 L_p(f^\alpha,g,k/2)=L_p(f/K,\Sigma^{(1)})(\omega).
 \]
\begin{theorem}
\label{thm:sigma1interpolationformula}
Let $\psi$ be a finite Hecke character of infinity type $(0,0)$ (in particular, it belongs to $\Sigma^{(1)}$) whose conductor divides $\ff p^\infty$. Let  $\psi_p$ be the $p$-adic avatar of $\psi$. If the conductor of $\psi$ is coprime to $p$ and $j$ is an integer such that $1\le j\le k-1$, then
  \[ L_p(f/K, \Sigma^{(1)})(\psi_{p}\cdot \chi_\cyc^{j-k/2}) = \frac{\cE(f,\psi,j)}{\cE(f)\cE^*(f)}\times\frac{i^{k-1}N^{2j-k+1}\Gamma(j)^2}{2^{2j+k-1}\pi^{2j}}\times \frac{L(f/K, \psi, j)}{ \langle f, f\rangle_{N}}  ,\]
  where  $\cE(f)=1-\lambda'/p\lambda$,  $\cE^*(f)=1-\lambda'/\lambda$,  $\lambda$  and $\lambda'$ are the roots of the Hecke polynomial of $f$ at $p$, with the former being the unit root and
  \[
  \cE(f,\psi,j)=\prod_{\fq\in\{\fp,\fp^c\}}\left(1-\frac{p^{j-1}}{\lambda\psi(\fq)}\right)\left(1-\frac{\lambda'\psi(\fq)}{p^{j}}\right).
  \]
  If the $p$-primary part of the conductor of $\psi$ is given by $\fp^m(\fp^c)^n$ with $m+n\ge 1$, then the value of the $p$-adic $L$-function at $\psi_p\cdot \chi_\cyc^{j-k/2}$ for $1\le j\le k-1$ is given by
  \[(-1)^{j-k/2}\times\frac{(Np^{m+n})^{j-1/2}}{\lambda^{m+n-1}}\times\frac{\cE(f,\psi,j)}{\cE(f)\cE^*(f)}\times\frac{i^{k}\Gamma(j)^2\tau(\psi)}{2^{2j+k-1}\pi^{k}}\times \frac{L(f/K, \psi, j)}{ \langle f, f\rangle_{N}}  ,\]
where $\tau(\psi)$ is root number of $\Theta(\overline{\psi_p})$.
\end{theorem}
\begin{proof}
We note that the first formula is given by \eqref{eq:interpolationHida} (on taking $\chi=1$ and $f_2$ the theta series corresponding to $\psi$). So, we suppose that $\Theta(\psi)$ is of level $Np^{s+1}$, and primitive at $p^{s+1}$ with $s\ge0$ (so that $s+1=m+n$, with the notation of the statement of our theorem). We first consider the case $1\le j<(k+1)/2$.

On employing the expansion \eqref{eq:eisenstein}, we may check that
\[
\cE_{1/N}(k/2-\psi_\QQ,k/2-1)=\d^{k/2-1}\left(\cE_{1/N}(-\psi_\QQ,0)\right),
\]
here $\d$ is the differential operator $q\cdot d/dq$. Similarly, we have 
\[
\Theta(\psi_p\cdot \chi_\cyc^{j-k/2})=\d^{j-k/2}\left(\Theta(\psi_p)\right).
\]
Then, by \cite[Lemma~6.5(iv)]{hida88} and \cite[Lemma~5.3]{hida85}, we have
\[
e_\ord(\Theta(\psi_p\cdot \chi_\cyc^{j-k/2})\times\cE_{1/N}(k/2-1-\psi_\QQ,k/2-1))=(-1)^{j-k/2}e_\ord\circ \Xi(\Theta(\psi)\times\delta_1^{j-1}\cE_{1/N}(-\psi_\QQ,0)),
\]
where $\Xi$ denotes the holomoprhic projection and $\delta_1^{j-1}$ is Shimura's differential operator.
Following the calculations in \cite[pages 224-225]{hidabook}, we have
\begin{align*}
&\left\langle \overline{ f^\lambda},e_\ord(\Theta(\psi)\times\cE_{1/N}(k/2-1-\psi_\QQ,k/2-1))\right\rangle_{N,p}\\
=&
\frac{(Np)^{k/2-1}\cdot p^{(k-1)s}}{(-1)^{j-k/2}\lambda^{s}}\times\left\langle\overline{W_{Np}\left(f^\lambda\right)}(p^sz), \Xi(\Theta(\overline{\psi_p})\times\delta_1^{j-1}\cE_{1/N}(\psi_\QQ,0))\right\rangle_{N,p^{s+1}},
\end{align*}
where we write $W_{Np^*}$ for the slash operator by the matrix $\begin{pmatrix}
0&-1\\
Np^*&0
\end{pmatrix}$ (with an appropriate weight, as normalized in \cite[page 6]{hida88}. This is different from the normalization in \cite[page 126]{hidabook}, which explains the factor $(Np)^{k/2-1}$).

Consider the identity
\[
W_{Np^{s+1}}\circ W_{Np}(h)=p^{ks/2}h(p^sz),
\]
where $h$ is a modular form of weight $k$ and level $Np$. This allows us to rewrite the inner product above as
\[
\frac{(Np)^{k/2-1}\cdot p^{(k/2-1)s}}{(-1)^{j-k/2}\lambda^{s}}\times\left\langle W_{Np^{s+1}}\circ W_{Np}^2\left(\overline{f^\lambda}\right), \Xi(\Theta(\overline{\psi_p})\times\delta_1^{j}\cE_{1/N}(\psi_\QQ,0))\right\rangle_{N,p^{s+1}}.
\]
But $W_{Np^{*}}^2(h)=(-1)^{l}h$ if $h$ is of weight $l$. This tells us that the quantity above is equal to
\[
\frac{(Np)^{k/2-1}\cdot p^{(k/2-1)s}}{(-1)^{j-k/2+1}\lambda^{s}}\times\left\langle W_{Np^{s+1}}\left(\overline{f^\lambda}\right), \Xi(W_{Np^{s+1}}^2(\Theta(\overline{\psi_p}))\times\delta_1^{j}\cE_{1/N}(\psi_\QQ,0))\right\rangle_{N,p^{s+1}}.
\]
As we have assumed that $1\le j<(k+1)/2$, we may apply \cite[Theorem~6.6]{hida88}, which is a modification of a classical result of Shimura, to deduce that this is equal to
\[
\frac{(Np)^{k/2-1}\cdot p^{(k/2-1)s}}{(-1)^{j-k/2+1}\lambda^{s}}\times\frac{i^{k-1}\Gamma(j)^2}{2^{2j+k-1}\pi^{k}(Np^{s+1})^{(k-2j-1)/2}}\times D_{Np}\left(f^\lambda,W_{Np^{s+1}} \left(\Theta(\overline{\psi_p})\right),j\right),
\]
where $D_{Np}$ is the Rankin-Selberg product as defined in \textit{loc. cit.}
We can simplify this expression as
\begin{equation}\label{eq:HidaRS}
\frac{(Np^{s+1})^{j-1/2}}{(-1)^{j-k/2+1}\lambda^s}\times\frac{i^{k-1}\Gamma(j)^2}{2^{2j+k-1}\pi^{k}}\times D_{Np}\left(f^\lambda,W_{Np^{s+1}} \left(\Theta(\overline{\psi_p})\right),j\right).
\end{equation}
 Furthermore,  we have the functional equation
\[
W_{Np^{s+1}}\left(\Theta(\overline{\psi_p})\right)=-i\tau(\psi_p)\times\Theta(\psi_p)
\]
(c.f. \cite[page 21]{PR}). Hence, we may further rewrite \eqref{eq:HidaRS} as
\[
\frac{(Np^{s+1})^{j-1/2}}{(-1)^{j-k/2}\lambda^s}\times\frac{i^{k}\Gamma(j)^2\tau(\psi)}{2^{2j+k-1}\pi^{k}}\times D_{Np}\left(f^\lambda,\Theta(\psi_p),j\right).
\]
Since the Euler factor of $f$ at $p$ agrees with that of $f^\lambda$, we may replace $f^\lambda$ by $f$ above, which gives the formula claimed.

For the case $(k+1)/2<j\le k-1$, the same proof goes through. The only difference is that, in order to apply \cite[Theorem~6.6]{hida88}, we would have to first of all make use of the functional equation of $\cE_{1/N}$ to replace it by its dual. This can be done using the calculations of Hida in \cite[page 328]{hidabook}.
\end{proof}

\begin{corollary}\label{cor:SU}
Let $\LSU$ be the $2$-variable $p$-adic $L$-function of Skinner-Urban defined in \cite[Theorem~12.7]{skinnerurbanmainconj} and $\psi$ as in Theorem~\ref{thm:sigma1interpolationformula}. Then, the quotient
\[\frac{\LSU(\psi_p\cdot\chi_\cyc^{k-2})}{L_p(f/K, \Sigma^{(1)})(\psi_{p}\cdot \chi_\cyc^{k/2-1})}
\]
is a $p$-adic unit if  \textup{\textbf{(H.Im.)}}, \textup{\textbf{(H.Dist.)}} and \textup{\textbf{(H.SS.)}}  hold.
\end{corollary}
\begin{proof}
This follows from the respective interpolation formulae and the fact that the root number has norm one whereas the norm of the Gauss sum is the square root of its conductor.
\end{proof}

We may also swap the role of $f$ and the theta series to define a second $p$-adic $L$-function, namely
\[
L_p(f/K,\Sigma^{(2)}):\lambda\mu\mapsto\frac{\left\langle\overline{g_\lambda},e_{\ord}\left(f\times\cE_{1/N}(\mu-k/2,-k/2-\lambda_\QQ-\mu)\right)\right\rangle}{\langle g_\lambda,g_\lambda\rangle}.
\]
where $\lambda$ is a Hecke character on $H_{\ff p^\infty}$ and $\mu$ factors through the norm map $H_{\ff p^\infty}\rightarrow H_{p^\infty}\rightarrow\Zp^\times$ and $g_\lambda$ is the corresponding CM eigenform corresponding to $\lambda$. However, the resulting function is no longer an element of $\Lambda_L(H_{\ff p^\infty})$, but rather its total ring of quotients. This is due to the Petersson inner products of the theta series (which range in the Hida family) that appear in the denominators of the values this function it ought to interpolate. The interpolation property reads as follows.

\begin{theorem}Let $\psi$ be a Hecke character in $\Sigma^{(2)}$ whose conductor divides $\ff$, with infinity-type $(a,b)$. Let $\psi_p$ for the $p$-adic avatar of $\psi$, then
 \[ L_{p}(f/K, \Sigma^{(2)})(\psi_{p}) =\frac{\cE(\psi,f,k/2)}{\cE(\psi)\cE^*(\psi)}\times\frac{i^{b-a-k+1}N^{a+b-k+3}b!(b-k)!}{2^{3b-a-k+3}\pi^{2b-k+3}}\times \frac{ L(f/K, \psi, k/2)}{ \langle g_{\lambda}, g_{\lambda}\rangle_{N_{g_\lambda}}}.\]
     where $g_{\lambda}$ is theta series corresponding to $\lambda = \psi |\cdot|^{-b}$ and $\cE(\psi,f,k/2),\cE(\psi),\cE^*(\psi)$ are some Euler factors.
\end{theorem}

We write $H_{\mathfrak{f}p^\infty}=\Delta\times\Gamma$, where $\Delta$ is a finite abelian  group. We end this subsection by introducing some distributions on $\Gamma$ that will play a role in the rest of this article, and which are obtained by restricting  $L_{p}(f/K, \Sigma^{(1)})$  and $L_{p}(f/K, \Sigma^{(2)})$.

\begin{defn} For each character $\alpha$ of $\Delta$, we define the subset $\Sigma^{(i,\alpha)}(\Gamma) \subset \Sigma^{(i)}$ (for $i=1,2$) to consist of those Hecke characters  whose $p$-adic avatars restrict on $\Delta$ to $\alpha$ and set  
$$\al_f^{(\alpha)} :=L_{p}(f/K, \Sigma^{(1)})\Big{|}_{\Sigma^{(1,\alpha)}(\Gamma)} \in \RR_L\,.$$
Fixing a topological generator $\gamma_\cyc$ of $\Gamma^\cyc$ and lifting it to $\Gamma$ (the chosen lift is still denoted by the same symbol), we may write
$$\al_f^{(\alpha)}=\al_{f,0}^{(\alpha)} + \al_{f,1}^{(\alpha)}\cdot \frac{\gamma_\cyc-1}{\log_p\chi_\cyc(\gamma_\cyc)}\mod (\gamma_\cyc-1)^2 $$
where $\al_{f,j}^{(\alpha)} \in \RR_L^\ac$ for $j=0,1$ and $\chi_\cyc$ is the cyclotomic character. Note that $\al_{f,0}^{(\alpha)}$ is simply the restriction of the two-variable $p$-adic $L$-function $\mathfrak{L}_f^{(\alpha)}$ to the anticyclotomic line and $\al_{f,1}^{(\alpha)}$ is the derivative of $\mathfrak{L}_f^{(\alpha)}$ with respect to the cyclotomic variable, and restricted to the anticyclotomic line.

In Section~\ref{subsec:twovarmainconj} (where we will appeal to the results of X. Wan in the indefinite case in order to sharpen some of the bounds we obtain in this article with the aid of Euler system machine) we will also need to consider the restriction 
$$\al_{f,\Sigma^{(2)}}^{(\alpha)} :=L_{p}(f/K, \Sigma^{(2)})\Big{|}_{\Sigma^{(2,\alpha)}(\Gamma)}\,$$
of $L_{p}(f/K, \Sigma^{(2)})$. 
\end{defn}
\subsection{Rankin-Selberg $L$-functions and root number calculations}
\label{subsec:rootnumbers}
In this section, we will review some basic facts concerning the sign of the root number for the Hecke $L$-functions of the Rankin-Selberg products of the form $f\times \theta_\xi$ where $\xi$ is a Hecke character of infinity type $(-m,m)$ and whose $p$-adic avatar factors through $\Gamma^\ac$. The following proposition (together with the ``minimalist philosophy") is the sign post in our work.
\begin{proposition}
\label{prop:rootresults}
\textup{i)} For every ring class character $\chi$ whose conductor is a sufficiently large power of $p$ we have 
$${\epsilon(f\otimes \chi/K)=\epsilon(f/K).}$$
If the prime $p$ does not divide $N$ or else $p$ splits in $K/\QQ$, this equality holds for every $\chi$. 
\\
\textup{ii)} For every Hecke character $\xi \in \Sigma_{\textup{cc}}^{(2)}$ of trivial conductor 
$${\epsilon(f\otimes \xi/K)=-\epsilon(f/K).}$$
\textup{iii)} Suppose that $p$ splits in $K/\QQ$ and fix a character $\xi \in \Sigma_{\textup{cc}}^{(2)}$ whose conductor is a power of $p$. Then for every ring class character $\chi$ with $p$-power conductor we have 
$${\epsilon(f\otimes \chi\xi/K)=-\epsilon(f/K).}$$
\end{proposition}
\begin{proof}
The first claim is \cite[Lemma 1.1]{cornutvatsal2007}. The second assertion may also be deduced from these calculations as follows (which we essentially reproduce from loc.cit. here for the convenience of the reader).

Let $\eta$ be the quadratic Hecke character associated to $K/\QQ$. For any Hecke character $\psi$, the sign of the global root number is given as the product $\epsilon(f\otimes \psi/K)=\prod_{\ell} \epsilon_\ell(f\times\theta_\psi)$ of local root numbers (where $\ell$ ranges over the places of $\QQ$). Following \cite[Section 1.1]{cornutvatsal2007},  we define the set
$$S(\psi):=\{\ell:\,\epsilon_\ell(f\times\theta_\psi)=-\eta_\ell(-1)\}$$
where $\eta_\ell$ is the local component of $\eta$. It is easy to see (using the product formula for the Hecke character $\eta$) that $\epsilon(f\otimes \psi/K)=(-1)^{\#S(\psi)}$. 

As explained in  \cite{cornutvatsal2007}, if a non-archimedean prime $\ell$ belongs to $S(\psi)$, then $\ell$ is necessarily inert and the local constituent $\pi_\ell$ of $\pi_f$ is either special or supercuspidal. Furthermore, since we assumed that $\xi$ has trivial conductor, it follows that an inert prime $\ell$ for which $\pi_\ell$ is supercuspidal belongs to $S(\xi)$ if and only if the $\ell$-adic valuation of the conductor of $\pi_\ell$ is odd. Retracing this paragraph for the trivial character, we conclude for a finite prime that $\ell \in S(\xi) \iff \ell \in S(\mathds{1})$. On the other hand, $\infty \in S(\mathds{1})$ but $\infty \notin S(\xi)$ (since $m\neq 0$, the local constituent $\pi(\xi)_\infty$ is non-trivial and $\pi_{f,\infty}$ is discrete series of weight $k\geq 2$). This shows that $\#S(\xi)=\#S(\mathds{1})-1$ and the proposition follows.

The proof of the third asserted identity is identical to the proof of the second since we readily know that $p \notin S(\chi\xi)$. 
\end{proof}
\begin{corollary}
\label{cor:cyclotomicderivative}
If $\epsilon(f/K)=-1$, then $\al_{f,0}^{(\alpha)}=0$ for every $\alpha$ as above.
\end{corollary}
\begin{proof}
This follows form the interpolation formulae in Theorem~\ref{thm:sigma1interpolationformula}.
\end{proof}
In view of Proposition~\ref{prop:rootresults}, one is led to the following prediction, conforming to the ``minimalist philosophy'':  
\begin{conj}
\label{conj:minimalism}
Assume that the eigenform $f$ is non-exceptional in the sense of \cite[Definition 1.3]{cornutvatsal2007}. For all but finitely many ring class characters $\chi$ with $p$-power conductor, the central critical values  $L(f,\chi,k/2)$ are non-vanishing if  the hypothesis (\textup{Sign} $+$) holds true. 

Suppose in addition that the prime $p$ splits in $K/\QQ$ and fix a character $\xi \in \Sigma^{(2)}_{\textup{cc}}$ whose conductor is a power of $p$. Then for all but finitely many ring class characters $\chi$ with $p$-power conductor, the values $L(f,\chi\xi,k/2)$ are non-vanishing if (\textup{Sign} $-$) holds true.
\end{conj}
\begin{remark}
\label{rem:whenexceptional}
The eigenform $f$ is called exceptional if $\pi _f\cong\pi_f\otimes\eta$. As explained in \cite{cornutvatsal2007}, $f$ is exceptional precisely when $f$ is the theta-series associated to a Hecke character $\psi$. In this case, one has a factorization
$$L(\pi_f\times\chi,s)=L(\psi\chi,s)\,L(\psi\chi^\prime,s)$$ 
where $\chi^\prime$ is the outer twist of $\chi$ by $\Gal(K/\QQ)$. Furthermore, both $L$-functions in the factorization of $L(\pi_f\times\chi,s)$ above have a functional equation with sign $\pm1 $. In the situations when both signs equal $-1$, in which case $L(\pi_f\times\chi,s)$ has at least a double zero at its central critical point, despite having $+1$ as its sign. This is the reason this case is excluded in Conjecture~\ref{conj:minimalism}. We note that the anticyclotomic Iwasawa theory in the exceptional case is precisely the subject of the articles \cite{agboolahowardordinary, arnoldhigherweightanticyclo}.
\end{remark}
The following two theorems are the best-known results (to our knowledge) towards Conjecture~\ref{conj:minimalism}. These are essential for the proofs of Theorem~\ref{thm:explicitreciprocity} (that guarantees the non-triviality of the restriction of the Beilinson-Flach Euler system restricted to the anticyclotomic line). The first is due to Chida and Hsieh~\cite[Theorem D]{chidahsiehcrelle}.
\begin{theorem}[Chida-Hsieh]
\label{thm:chidahsieh}
Suppose that $p^2\nmid N$ and the discriminant $\rm{disc}(K/\QQ)$ of $K$ is prime to $N$. Factor $N=p^{\alpha_p}N^+N^-$ where $N^+$ (resp., $N^-$) is only divisible by primes that are split (resp., inert) in $K$ and assume that $N^-$ is the square-free product of an odd number of primes. Suppose also that there exists a prime $\ell \nmid pN\rm{disc}(K/\QQ)$ with $\ell>k-2$ and such that the \textup{mod} $\ell$ representation associated to $f$ is absolutely irreducible. Then for all but finitely many ring class characters $\chi$ with $p$-power conductor, the central critical value  $L(f,\chi,k/2)$ is non-zero.
\end{theorem}
Note that the hypotheses on $N$ and on its factor $N^-$ ensure that the condition (Sign $+$) holds true. 

The second generic non-vanishing result is due to Hsieh~\cite[Theorem C]{hsiehnonvanishing}. We follow the notation we have set in the statement of Theorem~\ref{thm:chidahsieh}.

\begin{theorem}[Hsieh]
\label{thm:hsiehnonvanishing}
Suppose that $p\nmid N\rm{disc}(K/\QQ)$. Factor $N=N^+N^-$ where $N^+$ (resp., $N^-$) is only divisible by primes that are split (resp., inert or ramified) in $K$ and assume that $N^-$ is the square-free product of an even number of primes. Suppose also that there exists a prime $\ell \nmid pN\rm{disc}(K/\QQ)$ such that the \textup{mod} $\ell$ representation associated to $f$ is absolutely irreducible. Then for $\xi$ as in Proposition~\ref{prop:rootresults}(iii), the value  $L(f,\chi\xi,k/2)$ is non-zero for all but finitely many ring class characters $\chi$ with $p$-power conductor.
\end{theorem}
We remark that the hypotheses on $N$ and its factor $N^-$ ensure that the condition (Sign $-$) holds true. 

\section{Iwasawa theory}
\label{sec:casepord}
We suppose in this portion of our work that the eigenform $f$ is $p$-ordinary. We introduce the Selmer structures we will consider as part of this work. 

For any character $\alpha$ of $\Delta$ as above, we set $T_{f,\alpha}:=T_f\otimes{\alpha^{-1}}$. Let $F^+T_{f,\alpha} \subset T_{f,\alpha}$ denote the $G_{\QQ_p}$-stable \emph{Greenberg subspace} of $T_{f,\alpha}$. Set $\mathbb{T}_{f,\alpha}:=T_{f,\alpha}\otimes\LL^\iota$ and $F^+\mathbb{T}_{f,\alpha}=F^+T_{f,\alpha}\otimes\LL^\iota$ (where we allow the Galois groups act diagonally). 
\subsection{Selmer structures and Selmer groups}
Let $X$ stand for any quotient of $\mathbb{T}_{f,\alpha}:=T_{f,\alpha}\otimes\LL^\iota$ (such as $T_f$ itself, $\mathbb{T}_{f,\alpha}^{\ac}=T_{f,\alpha}\otimes\LL_{\ac}^\iota$ and $\mathbb{T}_{f,\alpha}^\cyc=T_{f,\alpha}\otimes\LL_{\cyc}^\iota$) and define $F^+X$ as the appropriate quotient of  $F^+\mathbb{T}_{f,\alpha}$.

\begin{defn}
For any $X$ as above, we define the Greenberg Selmer structure $\mathcal{F}_{\textup{Gr}}$ by setting
$$H^1_{\mathcal{F}_{\textup{Gr}}}(K_\nu,X)=\textup{im}\left(H^1(K_\nu,F^+X)\longrightarrow H^1(K_\nu,X) \right)$$
for $\nu=\mathfrak{p}$ and $\mathfrak{p}^c$ and demanding unramified local conditions away from $p$. We define the Selmer structure $\mathcal{F}_{+}$ by relaxing the Greenberg local condition at $\mathfrak{p}$ (so that the local condition for $\FFF_+$ on $X$ is given by 
$$H^1_+(K_p,X):=H^1_{\FFF_{\textup{Gr}}}(K_{\mathfrak{p}^c},X)\oplus H^1(K_{\fp},X)\subset H^1(K_p,X)$$ and likewise, we define $\mathcal{F}_{+^c}$ (and the submodule $H^1_{+^c}(K_p,X)\subset H^1(K_p,X)$) by relaxing the Greenberg local condition at the prime $\mathfrak{p}^c$ above $p$.

For any abelian extension $L/K$, we also define $H^1_+(L_p,T_{f,\alpha}) \subset H^1(L_p,T_{f,\alpha})$ as the collection of classes which are Greenberg at all primes of $L$ above $\fp^c$ (and which are not required to verify any further condition at primes above $\fp$) and similarly the submodule $H^1_{+^c}(L_p,T_{f,\alpha}) \subset H^1(L_p,T_{f,\alpha})$.
\end{defn}
\begin{remark}
Although the object of main interest would be $T_{f}$ (and its various deformations), it turns out that there is no additional difficulty if we treat instead all its twists $T_{f,\alpha}$ by anticyclotomic characters $\alpha$ of finite order. However, it also turns out that we are only able to handle only a certain collection of these twists, namely by those anticyclotomic characters $\alpha$ for which $\alpha(\fp)\neq \alpha(\fp^c)$. Although one expects that this condition is ultimately unnecessary, we are forced to work under this hypothesis in this article for two reasons. The first is that our method takes the Beilinson-Flach element Euler system as an input, and the very construction of the Euler system classes (which we recall in Section~\ref{subsubsec:BFelementsandplocalproperties} below) requires this condition. The second reason is that the proof of the two-variable main conjecture given by X. Wan (whose results we record and use in Section~\ref{subsec:twovarmainconj}) is also given under this hypothesis.
\end{remark}
\begin{defn}
\label{def:singularquotients}
For those $X$ as in the previous definition, we set 
$$H^1_{+/\f}(K_p,X):=H^1_+(K_p,X)\big{/}H^1_{\FFF_{\textup{Gr}}}(K_p,X)\,$$
and define the map
$$\textup{res}_{+/\f}: H^1_{\FFF_+}(K,X) \longrightarrow H^1_{+/\f}(K_p,X)$$
as the compositum of the natural restriction to a decomposition group above $p$, followed by the projection map.
\end{defn}
\begin{defn}
Given $X$ as above, we define the Selmer structure $\mathcal{F}_{-}$ by requiring cohomology classes be zero at $\mathfrak{p}$ and Greenberg at $\mathfrak{p}^c$. More precisely, the local condition for $\FFF_-$ on $X$ is given by 
$$H^1_-(K_p,X):=H^1_{\FFF_{\textup{Gr}}}(K_\mathfrak{p},X)\oplus \{0\}\subset H^1(K_p,X) \,.$$
We set 
$$H^1_{\f/-}(K_p,X):=H^1_{\FFgr}(K_p,X)\big{/}H^1_{\FFF_{-}}(K_p,X)\stackrel{\sim}{\lra}H^1_{\FFgr}(K_\mathfrak{p^c},X) \,$$
and define the map
$$\textup{res}_{\f/-}: H^1_{\FFgr}(K,X) \longrightarrow H^1_{\f/-}(K_p,X)$$
as the compositum of the natural restriction to a decomposition group above $p$, followed by the projection map.
\end{defn}

Given a Selmer structure $\mathcal{F}$ on $X$, we will consider the dual Selmer structure $\FFF^*$ on $X^\vee(1):=\textup{Hom}(X,\bbmu_{p^\infty})$, by setting for every prime $v$ of $K$ $$H^1_{\FFF^*}(K_v,X^\vee(1)):=H^1_{\FFF}(K_v,X)^\perp$$ 
under local Tate duality. Given $\mathcal{F}$, we will also define the Selmer group $H^1_\FFF(K,X)$ and $H^1_{\FFF^*}(K,X^\vee(1))$ in the usual way. 

\subsection{CM Hida families}
\label{subsubsec:CMhidafamilies}
Our aim in this section is to provide a discussion of CM Hida families following the discussion in \cite[\S\S3-5]{LLZ2}. All references in this section are to this article.
\subsubsection{Set up}
For a general modulus $\frak{n}$ of $K$, we let $H_{\frak{n}}$ denote the ray class group modulo $\frak{n}$ and let $K(\frak{n})$ denote the maximal $p$-extension contained in the ray class field modulo $\frak{n}$. We set $H_{\frak{n}}^{(p)}:=\Gal(K(\frak{n})/K)$.  For each ideal $\fm$ divisible by $\frak{f}$, we let $\LL_{\frak{m}}^{(p)}=\calO_L[H_{\frak{m}}^{(p)}]$ (this ring is denoted by $\LL_{\frak{m}}^{\frak{P}}$ in op. cit.).

We recall from the introduction that we have fixed an integral ideal $\frak{f}$ of $\calO_K$ that is prime to $p$ with $K(\frak{f})=K$. We have  also fixed a $p$-distinguished ring class character $\alpha$  modulo $\frak{f}$ in the sense that $\alpha(\fp)\neq \alpha(\fp^c)$; we assume that the local field $L$ is large enough to realize the character $\alpha$. We caution readers that our field $L$ is denoted by $L_{\frak{P}}$ in \textit{op. cit.} and their $\alpha$ is different from ours. 
For any ideal $\frak{m}=\fn\ff$ divisible by $\frak{f}$, we will regard $\alpha$ as a character of $\varprojlim H_{\frak{m}\fp^r}$ via the surjection $\varprojlim H_{\frak{m}\fp^r} \twoheadrightarrow H_\frak{f}$. 
   
   Let $\psi_0$ denote the unique Hecke character of $\infty$-type $(-1,0)$, conductor $\fp$ and whose associated $p$-adic Galois character factors through $\Gamma_\fp$ and set $\psi=\alpha\psi_0$.  (It is unique since the ratio of two such characters would have finite $p$-power order and conductor dividing $\fp$. This has to be the trivial character because of our assumption on the class number.) The Hecke character $\psi$ is of $\infty$-type $(-1,0)$ with conductor $\ff\fp$. Then the theta-series $$\Theta(\psi):=\sum_{(\frak{a},\ff \fp)=1}\psi(\frak{a})q^{\NN\frak{a}}\in S_2(\Gamma_1(D_KN_\chi p),\epsilon_K\omega^{-1})$$
  is the weight two specialization (with trivial wild character) of the CM Hida family with tame level $D_KN_\chi$ and character $\epsilon_K\omega$. The weight one specialization of this CM Hida family with trivial wild character equals the $p$-ordinary theta-series $\Theta^{\ord}(\alpha):=\sum_{(\frak{a},\ff \fp)=1}\alpha(\frak{a})q^{\NN\frak{a}}\in S_{1}(\Gamma_1(D_KN_\chi p),\epsilon_K)$ of $\alpha$. We finally let $\psi_L:G_K\lra L^\times$ denote the Galois character associated to the $p$-adic avatar of $\psi$, which is obtained via the geometrically normalized Artin map.

 \begin{defn}
We let $\mathcal{N}$ denote the collection of square-free integral ideals of $\mathcal{O}_K$ which are prime to $\ff p$. In what follows we will tacitly denote the moduli which are divisible by $\mathfrak{f}$ but prime to $p$ by $\mathfrak{m}$, and those which are prime to $\mathfrak{f}p$ and square free by $\mathfrak{n}$.
\end{defn}
\subsubsection{Hecke algebras and group rings}
 Set $M:=D_K\NN\frak{m}$ and let ${\bf{T}}_{Mp}$ denote the Hecke algebra given as at the start of \cite[\S4.1]{LLZ2} and define the maximal ideal $\mathcal{I}_{\frak{mp}}$ of ${\bf{T}}_{Mp}$ as in Definition 5.1.1 of op.cit. We remark that in order to determine the map $\phi_{\frak{m}\fp}$ that appears in this definition, we use the algebraic Hecke character $\psi$ we have chosen above. It follows by Prop. 5.1.2 that $\mathcal{I}_{\frak{mp}}$ is non-Eisenstein, $p$-ordinary and $p$-distinguished. By Theorem~4.3.4 of loc. cit., the ideal $\mathcal{I}_{\frak{mp}}$ corresponds uniquely to a $p$-distinguished maximal ideal $\mathcal{I}$ of the universal ordinary Hecke algebra ${\bf{T}}_{Mp^\infty}$ acting on $H^1_{\ord}(Y_1(Mp^\infty))$ (definitions of these objects may be found in Definition 4.3.1 of loc. cit.). The said correspondence is induced from Ohta's control theorem~\cite[Theorem 1.5.7(iii)]{ohta99}, which also associates $\mathcal{I}_{\fm\fp}$ a unique non-Eisenstein, $p$-ordinary and $p$-distinguished maximal ideal $\mathcal{I}_{\fm\fp^r}$ of ${\bf{T}}_{Mp^r}$ for each $r\geq 1$ (which is easily seen to coincide with the kernel of the compositum of the arrows  ${\bf{T}}_{Mp^r}\stackrel{\phi_{\fm\fp^r}}{\lra}\calO_L[H_{\fm\fp^r}]\ra\calO_L\ra \calO_\varpi$, and therefore with its original form given in \cite[Definition 5.1.1]{LLZ2}).   The ideal $\mathcal{I}$ determines a CM Hida family that we shall denote by $\g_{\frak{m}}$, whose associated Galois representation $R_{\g_{\frak{m}}}^*$ is $H^1_{\ord}(Y_1(Mp^\infty))_{\mathcal{I}}$. When $\fm=\ff$, note that $\g_\ff$ is the CM Hida family of tame level $D_KN_\alpha$ and character $\epsilon_K\omega$ we have discussed above.
 
\subsubsection{Galois representations attached to CM branches of Hida families and patching}
  The discussion in the following paragraph is based on \cite[\S5.1]{LLZ2}. The maps $\phi_{\fm\fp^r}$ induce morphisms $({\bf{T}}_{Mp^r})_{\mathcal{I}_{\fm\fp^r}}\stackrel{\phi_{\fm\fp^r}}{\lra}\calO_L[H_{\fm\fp^r}^{(p)}]$, which are compatible as $r$ varies (thanks to the choice of the $p$-ordinary maximal ideals ${\mathcal{I}_{\fm\fp^r}}$) and gives rise in the limit to a map
 $$\phi_{\fm,\fp^\infty}:\,({\bf{T}}_{Mp^\infty})_{\mathcal{I}}\lra \varprojlim_r \calO_L[H_{\fm\fp^r}^{(p)}]=:\LL_{\fm\fp^\infty}\,.$$
 On the other hand, the $G_\QQ$-representations $H^1(\psi,\fm\fp^r)$ are defined by setting 
 $$H^1(\psi,\fm\fp^r):= H^1(Y_1(Mp^r))_{\mathcal{I}_{\fm\fp^r}}\,\otimes_{\phi_{\fm\fp^r}}\calO_L[H_{\fm\fp^r}^{(p)}]\,.$$
 These are denoted by $H^1(\psi,\frak{mp}^r,\frak{P})$ in \cite[Definition~5.2.2]{LLZ2} and free of rank $2$ over $\calO_L[H_{\fm\fp^r}^{(p)}]$ (c.f. Proposition 5.2.3 in \textit{op. cit.}). On passing to limit in $r$, we have 
 $$H^1(\psi,\fm\fp^\infty)=H^1_\ord(Y_1(M))_\mathcal{I}\otimes_{\phi_{\fm\fp^\infty}}\LL_{\fm\fp^\infty}$$ 
We define a family of morphisms of Galois modules 
$$\pi_{\frak{m}}: R_{\Bg_{\fm}}^*=H^1_\ord(Y_1(M))_\mathcal{I}\lra H^1(\psi,\fm\fp^\infty)$$  
given simply by $r\mapsto r\otimes 1$ and which are compatible as $\fm$ varies.
 
 \begin{remark}
It is assumed at the start of \cite[\S 5.1]{LLZ2} that $\psi$ is a an algebraic Hecke character whose conductor is prime to $p$. However, as we are assuming $p$ splits in $K$ and that $\alpha$ is $p$-distinguished, Remark 5.1.3 of \textit{op. cit.} still applies and tells us that the same results hold true in our setting.
\end{remark}

The following statement is a very slight extension\footnote{There is a typo in the statement of Corollary 5.2.6: In order to apply the previous results in \S5, the set of ideals considered should be the ones coprime to $\fp^c$, not $\fp$.} of Corollary 5.2.6 via Proposition 5.2.5 of \textit{op. cit.}: 
\begin{proposition}
\label{prop:LLZ2Cor526enhanced}
There exists a family of isomorphisms
$$\nu_{\frak{m},r}\,:\,H^1(\psi,\frak{mp}^r)\stackrel{\sim}{\lra} \textup{Ind}_{K(\frak{mp}^r)}^\QQ\,\calO_L(\psi^{-1}_L)$$
of  $\LL_{\frak{mp}^r}^{(p)}[[G_\QQ]]$-modules such that the diagram
$$\xymatrix{
H^1(\psi,\frak{m}^\prime\frak{p}^r)\ar[d]_{\mathcal{N}^{\frak{m}^\prime\fp^r}_{\frak{m}\fp^s}}\ar[r]^(.42){\nu_{\frak{m}^\prime,r}}_(0.42)\sim& \textup{Ind}_{K(\frak{m}^\prime\frak{p}^r)}^\QQ\,\calO_L(\psi^{-1}_L)\ar[d]\\
H^1(\psi,\frak{m}\frak{p}^s)\ar[r]^(.42){\nu_{\frak{m},s}}_(0.42)\sim& \textup{Ind}_{K(\frak{m}\frak{p}^s)}^\QQ\,\calO_L(\psi^{-1}_L)
}$$
commutes as $\frak{m}\mid \frak{m}^\prime$ range over integral ideals of $\calO_K$ that are divisible by $\frak{f}$ and coprime to $p$; and $r\geq s$ over non-negative integers. 
\end{proposition}
Here, there vertical map on the right is the obvious map induced from
$$\Gal(K(\frak{m}^\prime\fp^r)/K)\ra \Gal(K(\frak{m}\fp^s)/K)$$ 
and $\mathcal{N}^{\frak{m}^\prime\fp^r}_{\frak{m}\fp^s}$ is the norm map which is given (together with its fundamental property) in Proposition 5.2.5 of \textit{op. cit.} 

\begin{corollary}
\label{cor:nun}
If $\frak{m}$ is an ideal divisible by $\frak{f}$ and coprime to $p$, there exists a family of isomorphisms
$$\nu_{\frak{mp}^r}^{(j)}: H^1\left(\QQ, U_{f}^*(-j)\otimes H^1(\psi,\frak{mp}^r)\widehat\otimes\,\LL_\cyc^\iota\right)\stackrel{\sim}{\lra} H^1\left(K(\frak{mp}^r), U_{f}^*(-j)(\psi_L^{-1})\,\otimes\,\LL_{\fp^c}^\iota\right)$$
which are compatible as $\frak{m}$ and $r$ varies (in the sense that they induce a commutative diagram, analogous to that in Proposition~\ref{prop:LLZ2Cor526enhanced}). Here, $\LL_{\fp^c}:=\ZZ_p[[\Gamma_{\fp^c}]]$ and $\Gamma_{\fp^c}$ is the Galois group of the $\ZZ_p$-extension of $K$ unramified outside $\fp^c$.
\end{corollary}
\begin{proof}
This is a consequence of Shapiro's Lemma and Proposition~\ref{prop:LLZ2Cor526enhanced}, by passing to limit in $r$.
\end{proof}

We set $U_{f,\alpha}:=U_f\otimes\alpha$, so that $T_{f,\alpha}=U_{f,\alpha}^*(1-k/2)$. Let $W_{f,\alpha}=U_{f,\alpha}\otimes\QQ_p$.
\begin{corollary}
\label{cor:topropLLZ2Cor526enhanced}
Let $\fm$ be as in Corollary~\ref{cor:nun}.
There exists a family of morphisms
$$\mu_{\frak{m}}^{(j)}:\varprojlim_r H^1\left(\QQ, U_{f}^*(-j)\otimes H^1(\psi,\frak{mp}^r)\,\otimes\,\LL_\cyc^\iota\right){\lra}  H^1\left(K(\frak{m}), U_{f,\alpha}^*(-j)\,\otimes\,\LL^\iota\right)$$
which are compatible as $\frak{m}$ varies (in the evident sense).
\end{corollary}
\begin{proof}
The maps $\mu_\frak{m}^{(j)}$ are obtained by composing ${\nu}_{\fm}^{(j)}$ from Corollary~\ref{cor:nun} with the compositum of the arrows
\begin{align*}
\varprojlim_r H^1\left(K(\frak{mp}^r),  U_{f}^*(-j)(\psi_L^{-1})\,\otimes\,\LL_{\fp^c}^\iota\right) &\stackrel{\textup{cor}}{\lra}\varprojlim_r H^1\left(K(\frak{m})K({\fp}^r),U_{f}^*(-j)(\psi_L^{-1})\,\otimes\,\LL_{\fp^c}^\iota\right)\\
&\lra H^1\left(K(\frak{m}), U_{f}^*(-j)(\psi_L^{-1})\,\otimes\,\LL^\iota\right)
\\&\lra H^1\left(K(\frak{m}), U_{f,\alpha}^*(-j)\,\otimes\,\LL^\iota\right)\\
\end{align*}
where $\textup{cor}=\textup{cor}_{K(\frak{m})K({\fp}^r)}^{K(\frak{mp}^r)}$ is the corestriction map, the second arrow is deduced from Shapiro's lemma, the last arrow is induced from the fact that $\alpha^{-1}\psi_L$ factors through $\Gamma$ and their compositum have the desired compatibility as $\frak{m}$ varies since each of these arrows does.
\end{proof}
The family of morphisms $\mu_{\frak{m}}^{(j)}$ may be rewritten in a more compact form as
\begin{equation}
\label{eqn:compatiblemorphismsmufrakn}
\mu_{\frak{m}}^{(j)}: H^1\left(\QQ, U_{f}^*(-j)\otimes H^1(\psi,\frak{mp}^\infty)\,\otimes\,\LL^\iota_
\cyc\right){\lra}  H^1\left(K(\frak{m}), U_{f,\alpha}^*(-j)\,\otimes\,\LL^\iota\right)
\end{equation}
where we have set $H^1(\psi,\frak{mp}^\infty):=\varprojlim_r H^1(\psi,\frak{mp}^r)$.


\subsection{Beilinson-Flach elements and their $p$-local properties} 
\label{subsec:BFlocally}
Until the end of Section~\ref{subsec:rubinsformula}, the hypotheses \textup{\textbf{(H.Im.)}}, \textup{\textbf{(H.Dist.)}} and \textup{\textbf{(H.SS.)}} are in effect.
\subsubsection{Belinson-Flach elements and Coleman maps for Rankin-Selberg convolutions}
\label{subsubsec:BFelementsandplocalproperties}
Let $f_1$ and $f_2$ be two cuspidal new eigenforms of weights $k_1$ and $k_2$, levels $N_1$ and $N_2$, characters $\varepsilon_1$ and $\varepsilon_2$ respectively.  We assume that $p$ is coprime to $N_1N_2$ and that both $f_1$ and $f_2$ are ordinary at $p$. For integers $m\ge1$ and $c>1$ that is coprime to $6N_1N_2$, there exists a generalized Beilinson-Flach element
\[
_c\BF^{f_1,f_2}_{m,j}\in H^1(\QQ(\mu_{m}),W_{f_1}^*\otimes W_{f_2}^*(-j))
\]
as constructed in \cite{KLZ1}. Here, $W_{f_i}^*$ denotes the linear dual $\Hom(W_{f_i},\mathfrak{o}_L)$. These elements satisfy an Euler-system type norm relations as $m$ varies.

Let $\Bf_1$ and $\Bf_2$ be two Hida families of tame levels $N_1$ and $N_2$ respectively. We fix an integer $N$ that is divisible by $N_1$ and $N_2$. Let $\Lambda_{\Bf_1}$ and $\Lambda_{\Bf_2}$ be the localizations of the Hecke algebra $\mathbf{T}_{Np^\infty}$ at $\Bf_1$ and $\Bf_2$ respectively. We write $W_{\Bf_i}$  for the corresponding $\Lambda_{\Bf_i}$-adic representations. We may extend the definition of $F^\pm W_{f_i}$ to $F^\pm W_{\Bf_i}$. The Beilinson-Flach elements we described above deform to an element 
\[
_c\BF^{\Bf_1,\Bf_2}_{m}\in H^1(\QQ(\mu_{m}),W_{\Bf_1}^*\,\widehat\otimes\, W_{\Bf_2}^*\,\widehat\otimes\, \Lambda_\cyc(-\Bj)),
\]
where $\Bj$ is the canonical character $\Gamma^\cyc\rightarrow\Lambda_\cyc^\times$ (c.f. \cite[Definition~8.1.1]{KLZ2}).

As in Definition~8.2.1 in \textit{op. cit.}, we write 
\[
\DD(M)=\left(M\,\widehat\otimes\,_{\Zp}\widehat{\ZZ}_p^{\ur}\right)^{G_{\Qp}}
\]
for a unramified $p$-adically complete $\Zp[G_{\Qp}]$-module $M$. Then there is an injective morphism of $\Lambda_{\Bf_1}\,\widehat\otimes\,\Lambda_{\Bf_2}\,\widehat\otimes\,\Lambda_\cyc$-modules
\[
\calL:H^1(\Qp,F^-W_{\Bf_1}^*\,\widehat\otimes\, F^+W_{\Bf_2}^*\,\widehat\otimes\,\Lambda_\cyc(-\Bj))\rightarrow \DD(F^-W_{\Bf_1}^*\,\widehat\otimes\, F^+W_{\Bf_2}^*)\,\widehat\otimes\,\Lambda_\cyc
\]
given by Theorem~8.2.8 in \textit{op. cit.}.

We fix  $\Ba$ to be a new and non-Eisenstein branch of $\Bf_1$. Let $I_\Ba$ be the congruence ideal of Hida from \cite{hida88}. Let $\omega_{\Bf_1}$ and $\eta_\Ba$ be the two maps defined in Proposition~9.6.2 of \cite{KLZ2}. Then, we have the pairing
\[
\langle-,\eta_\Ba\otimes\omega_{\Bf_2}\rangle:\DD(F^-W_{\Bf_1}^*\,\widehat\otimes\, F^+W_{\Bf_2}^*)\,\widehat\otimes\, \Lambda_\cyc\ra(I_\Ba\,\widehat\otimes\,\Lambda_{\Bf_2}^{\textup{cusp}}\,\widehat\otimes\,\Lambda_\cyc)\otimes_{\Zp}\Zp[\mu_N]
\]
which is a $\Lambda_\Bf\,\widehat\otimes\,\Lambda_{\Bf_2}\,\widehat\otimes\,\Lambda_\cyc $-morphism. On combining this with $\calL$, we have the map
\begin{equation}\label{eq:colemanRS}
\langle\calL,\eta_\Ba\otimes\omega_{\Bf_2}\rangle:H^1(\Qp,F^-W_{\Bf_1}^*\,\widehat\otimes\, F^+W_{\Bf_2}^*\,\widehat\otimes\,\Lambda_\cyc(-\Bj))\rightarrow(I_\Ba\,\widehat\otimes\,\Lambda_{\Bf_2}^{\textup{cusp}}\,\widehat\otimes\,\Lambda_\cyc)\otimes_{\Zp}\Zp[\mu_N]
.
\end{equation}
From \cite[Theorem B]{KLZ2} (which is called the explicit reciprocity law), we have 
\begin{equation}\label{eq:reciprocity}
\langle\calL(\BF_1^{\Bf_1,\Bf_2}),\eta_\Ba\otimes\omega_{\Bf_2}\rangle=L_p(\Ba,\Bf_2,1+\Bj),
\end{equation}
where $L_p(\Ba,\Bf_2,1+\Bj)$ is the specialization of the $p$-adic $L$-function $L_p(\Bf_1,\Bf_2,1_\Bj)$ defined in \S\ref{S:padicL},  after dispensing  with the parameter $c$ as in \S10.3 of \textit{op. cit.}.

\subsubsection{Beilinson-Flach Euler system over $K$}
\label{subsubsec:ESoverK}
We are now ready to introduce the collection of Beilinson-Flach elements over the imaginary quadratic field $K$. For each modulus $\mathfrak{m}=\mathfrak{n}\ff$ with $\mathfrak{n}\in \mathcal{N}$, we set $\Bf_2$ to be the CM Hida family $\Bg_{\mathfrak{m}}$ we have introduced above.

 We let $\mathbf{f}_1$ denote the Hida family carrying the $p$-ordinary stabilization of $f$ as a weight-$k$ specialization. Specializing $\Bf_1$ at $f$ and making use of the isomorphism $\mu_{\fm}^{(j)}$ together with the projection $\pi_{\fm}$, Beilinson-Flach elements give rise to a class
\begin{align*}
\BF^{f}_{\mathfrak{n},j,\alpha}\in H^1(K(\mathfrak{n}),U_{f,\alpha}^*(-j)\otimes\LL^\iota)
\end{align*}
for each $\mathfrak{n}$ and $j$.  

\begin{proposition}
The Beilinson-Flach elements $\BF^{f}_{\mathfrak{n},j,\alpha}$ satisfy an Euler-system distribution relation as $\mathfrak{n}$ varies. More precisely, for every $\mathfrak{nl}\in \mathcal{N}$ we have
\begin{equation}
\label{eqn:BFdistrrelation}
\textup{cor}_{{K(\fl\mathfrak{n})}/{K(\mathfrak{n})}}\left(\BF^{f}_{\mathfrak{nl},j,\alpha}\right)=P_\fl([\fl])\cdot \BF^{f}_{\mathfrak{n},j,\alpha}\,.
\end{equation}
Here, $P_\fl(X)$ denotes the $\fl$-local Euler polynomial for $W_{f,\alpha}(1+j)$ and $[\fl]$ denotes the image of the class of  $\fl$ in $H_{\mathfrak{nf}}$ in $\LL^{\frak{o}}[\textup{Gal}(K(\fn)/K)]$, obtained via the surjection $H_{\mathfrak{nf}}\twoheadrightarrow \textup{Gal}(K(\fn)/K)$ induced by the geometrically normalized Artin map.

Furthermore, the image of $\BF^{f}_{\mathfrak{n},j,\alpha}$ under the localization map at ${\fp^c}$ falls inside the image of $H^1(K_{{\fp}^c},F^+U_{f,\alpha}^*(-j)\otimes\LL^\iota)$.
\end{proposition}
\begin{proof}
The asserted distribution relation (\ref{eqn:BFdistrrelation}) is analogous to \cite[Theorem~3.5.1]{LLZ2}. It follows once we formally incorporate the arguments of \cite{LLZ2} (that were used to prove Theorem 3.5.1 in \emph{op. cit.}; more precisely Lemma A.2.1, Theorem A.3.1 and Theorem A.4.1 therein) with the more general constructions of \cite{KLZ2}. For this reason and to keep the length of this article within reasonable length, we will only explain how various norm relations proved in \cite{KLZ2} map out with the ingredients (Lemma A.2.1 in \cite{LLZ2} and Theorems 3.1.1 and Theorem 3.3.1 in \cite{LLZ1}) that were utilized in order to prove \cite[Theorem 3.5.1]{LLZ2}. 

In this proof, we will very tacitly make use of the notation introduced in \cite[\S 5]{KLZ2} without providing their proper definitions here, but giving a precise reference to \emph{op. cit.} Let $_c\mathcal{RI}^{[j]}_{M,N,a}$ denote the Rankin-Iwasawa class defined as in Definition 5.1.5 of \textit{op. cit.} Following \cite[Definition~2.1.1]{LLZ2}, we may define the \emph{asymmetric Rankin-Iwasawa class} 
$$_c\mathcal{RI}^{[j]}_{M,N,N^\prime,a} \in H^3_{\textup{\'et}}\left(Y(M,N)\times Y(M,N^\prime),\LL(\mathscr{H}_{\ZZ_p}\langle t_N\rangle)^{[j]}\boxtimes \LL(\mathscr{H}_{\ZZ_p}\langle t_{N^\prime}\rangle)^{[j]}  (2-j)\right)$$ 
(for integers $N,N^\prime \geq 5$) as the image of $_c\mathcal{RI}^{[j]}_{M,R,a}$ (for some $R$ divisible by $N$ and $N^\prime$, having  the same prime factors as $N, N^\prime$) induced from the obvious degeneracy maps 
$$Y(M,R)^2 \lra Y(M,N) \times Y(M,N^\prime)$$
and morphism of sheaves given by
$$\xymatrix{\LL(\mathscr{H}_{\ZZ_p}\langle t_N\rangle)^{[j]}\boxtimes \LL(\mathscr{H}_{\ZZ_p}\langle t_{N^\prime}\rangle)^{[j]}\ar[rrr]^(.6){[R/N]_*\,\boxtimes\,\, [R/N^\prime]_*}&&&\LL(\mathscr{H}_{\ZZ_p}\langle t_R\rangle)^{[j,j]}}\,.$$ 
Here, the maps $[d]_*$ (for $d\in \ZZ$) are given as in \cite[Notation 5.1.4]{KLZ2}. It follows from \cite[Theorem 5.3.1]{KLZ2} that this element is independent of the choice of $R$. The asymmetric Rankin-Iwasawa classes $_c\mathcal{RI}^{[j]}_{M,N,N^\prime,a}$ will play the role of the asymmetric zeta elements denoted by $_c\mathcal{Z}(m,N, N^\prime,j)$ in \cite{LLZ2}. 

Likewise, we may define the asymmetric $\LL$-adic Beilinson-Flach elements
$$_c\mathcal{BF}^{[j]}_{M,N,N^\prime,a} \in H^3_{\textup{\'et}}\left(Y_1(N)\times Y_1(N^\prime)\times \mu_M^\textup{o}, \LL(\mathscr{H}_{\ZZ_p}\langle t_N\rangle)^{[j]}\boxtimes \LL(\mathscr{H}_{\ZZ_p}\langle t_{N^\prime}\rangle)^{[j]}  (2-j)\right)$$
for $N,N^\prime$ as above. The asymmetric $\LL$-adic Beilinson-Flach classes $_c\mathcal{BF}^{[j]}_{M,N,N^\prime,a}$ will play the role of the elements denoted by $_c{\Xi}(m,N, N^\prime,j)$ in \cite{LLZ2}.

Using the commutative diagrams  
$$\xymatrix{Y(M,\ell R)^2 \ar[r]\ar[d]_{\textup{pr}_1\times \,\textup{pr}_2}&Y(M,N)\times Y(M,\ell N^\prime)\ar[d]^{1\times \textup{pr}_2}\\
Y(M, R)^2\ar[r]& Y(M,N)\times Y(M,N^\prime)
}$$
and 
$$\xymatrix{ \LL(\mathscr{H}_{\ZZ_p}\langle t_{\ell R}\rangle)^{[j,j]} &&&\ar[lll]_(.57){[\ell R/N]_*\,\boxtimes\,\, [R/N^\prime]_*} \LL(\mathscr{H}_{\ZZ_p}\langle t_N\rangle)^{[j]}\boxtimes \LL(\mathscr{H}_{\ZZ_p}\langle t_{\ell{N}^\prime}\rangle)^{[j]}\\
\LL(\mathscr{H}_{\ZZ_p}\langle t_{R}\rangle)^{[j,j]} \ar[u]^{[\ell]_*\boxtimes\,[\ell]_*}&&&\ar[lll]_(.57){[R/N]_*\,\boxtimes\,\, [R/N^\prime]_*} \LL(\mathscr{H}_{\ZZ_p}\langle t_N\rangle)^{[j]}\boxtimes \LL(\mathscr{H}_{\ZZ_p}\langle t_{{N}^\prime}\rangle)^{[j]}\ar[u]_{1\boxtimes\,[\ell]_*}
}$$
(the first of which also appears on page 1618 of \cite{LLZ2}), the proof of \cite[Theorem A.3.1]{LLZ2} formally adapts to deduce the analogous norm relations (corresponding to the degeneracy maps $1\times \textup{pr}_1$ and $1\times \textup{pr}_2$) for the asymmetric Ranking-Iwasawa classes we introduced above, from those for the symmetric Rankin-Iwasawa classes (Theorem 5.5.1 in \cite{KLZ2}). One may now formally adapt the proof of \cite[Theorem A.4.1]{LLZ2} in order to deduce the corresponding norm relations for the $\Lambda$-adic Beilinson-Flach elements (which are also induced from the degeneracy maps $1\times \textup{pr}_1$ and $1\times \textup{pr}_2$) from those for the Rankin-Iwasawa classes obtained as above. 

We now conclude with the proof of the distribution relation for the collection $\{\BF^{f}_{\mathfrak{n},j,\alpha}\}_{\mathfrak{n}\in \mathcal{N}}$ using Corollary~\ref{cor:topropLLZ2Cor526enhanced} and arguing as in the proof of \cite[Theorem~3.5.1]{LLZ2}. Note that \cite[Theorem~3.5.1]{LLZ2} is a formal consequence of the norm relations (corresponding to the degeneracy maps $1\times \textup{pr}_1$ and $1\times \textup{pr}_2$) for the asymmetric $\Xi$'s of loc. cit. that were proved in Theorem A.4.1 of the same article; which we replace with the norm relation for the $\mathcal{BF}$'s we may obtain as above. 
 
 The second assertion concerning the local image of the Beilinson-Flach elements is an immediate consequence of \cite[Proposition~8.1.7]{KLZ2}.
 \end{proof}

We fix from now on a choice of $\alpha$ and  omit it from the notation. Given a finite  extension $K'$ over $K$ that is contained in side $K_\infty$, we write $\BF^{f}_{K',\mathfrak{n},j}$ for the image of  $\BF^{f}_{\mathfrak{n},j}$ in $H^1(K',U_{f,\alpha}^*(-j))$ under the natural projection map. When $\mathfrak{n}=1$, we write $\BF^{f}_{K',j}$ in place of $\BF^{f}_{K^\prime,1,j}$. We shall also write $\BF^{f}_{D_\infty,j}$ and $\BF^{f}_{K_\infty,j}$ for the inverse limit of the classes $\BF^{f}_{K',j}$ as $K'$ runs through all finite extensions of $K$ that are contained inside $D_\infty$ and $K_\infty$, respectively.

Let $\Ba$ denote the branch of $\Bf_1$ that passes through $f$. Note that this is non-Eisenstein in the sense of \cite[Definition~7.6.2]{KLZ2} because of the hypothesis \textbf{(H.Im.)}.  After enlarging $L$ if necessary, the map in \eqref{eq:colemanRS} gives us
\[
\col^{(1,\alpha)}:H^1(K_\fp,W_{f,\alpha}^*/F^+W_{f,\alpha}^*(1-k/2)\otimes\LL^{\iota})\rightarrow \Lambda_{L}.
\]
The first reciprocity law for Beilinson-Flach elements tells us that 
\begin{equation}
\label{eqn:firstreciprocitylaw}
\col^{(1,\alpha)}\left(\BF^{f}_{K_\infty,k/2-1}\right)=\mathfrak{L}_{f}^{(\alpha)}\,.
\end{equation} 
On the other hand, take $\Bf_1$ to be the Hida family of CM theta series $\Bg_{\mathfrak{m},\cP}$ above and $\Bf_2$ a Hida family that passes through $f$, then we may specialize $\langle\calL,\eta_\Ba\otimes\omega_{\Bf_2}\rangle$ to obtain another Coleman map 
\[
\col^{(2,\alpha)}:H^1(K_{\fp^c},F^+W_{f,\alpha}^*(1-k/2)\otimes\LL^{\iota})\ra I_{\Ba}\,\widehat\otimes\, \Lambda_\cyc\otimes_{\Zp}\Zp[\mu_N].
\]
 The second reciprocity law for the Beilinson-Flach elements reads 
\begin{equation}
\label{eqn:secondreciprocitylaw}
\col^{(2,\alpha)}\left(\BF^{f}_{K_\infty,k/2-1}\right)=\mathfrak{L}_{f,\Sigma^{(2)}}^{(\alpha)}\,.
\end{equation}

Suppose that furthermore that the  CM series  satisfies both \textbf{(H.Im.)} and \textbf{(H.Dist.)}. By \cite{HT}, there exists an anticyclotomic projection of Katz' $p$-adic $L$-function $L^{\textup{Katz}}_{\Ba}\in \RR_\Phi$ such that $I_\Ba^{-1}=(L^{\textup{Katz}}_{\Ba})$ (c.f., \cite[Definition~7.6]{xinwanwanrankinselberg}). We may therefore define 
\[
\widetilde{\col}^{(2,\alpha)}=L^{\textup{Katz}}_{\Ba}\times \col^{(2,\alpha)}.
\]
This image of this map now falls inside $\RR_\Phi$.  Furthermore, we  have
\begin{equation}\label{eq:WanLfunction}
\widetilde{\col}^{(2,\alpha)}(\BF_{K_\infty,1})=\calL_{f,\alpha}^{\textup{Hida}}, 
\end{equation}
where $\calL_{f,\alpha}^{\textup{Hida}}$ is Wan's $p$-adic $L$-function defined as in \textit{loc. cit.}

We note that we may scale these  maps by a power of a uniformizer of $\mathfrak o$ so that both maps takes values in $\Lambda^{\mathfrak{o}}$. Note also that the cokernel of both $\col^{(1)}$ and $\widetilde\col^{(2)}$ are pseudo-null  (see the discussion in the last paragraph before Theorem 10.6.4 of \cite{KLZ2}).

For $j=k/2-1$, we henceforth drop both $f$ and $j$ from the notation and write simply $\BF_{K',\mathfrak{m}}$, $\BF_{K'}$, $\BF_{K_\infty}$, $\BF_{D_\infty}$, $\BF_{K'}$ for the corresponding Beilinson-Flach elements for $T_{f,\alpha}$. 
\subsubsection{Non-triviality of the anticyclotomic Beilinson-Flach Euler system}
Factor $N=p^{\alpha_p}N^+N^-$ where $N^+$ (resp., $N^-$) is only divisible by primes that are split (resp., inert or ramified) in $K$ and \emph{assume until the end of Section~\ref{subsec:rubinsformula} that $N^-$ is square-free}. Suppose also that there exists a prime $\ell \nmid pN\rm{disc}(K/\QQ)$ with $\ell>k-2$ and such that the \textup{mod} $\ell$ representation associated to $f$ is absolutely irreducible. 

Note that (Sign $+$) is verified if and only if $N^-$ is a product of odd number of primes; otherwise  (Sign $-$) holds true. In either case, Theorem~\ref{thm:explicitreciprocity} below equips us with an Euler system for $T_{f,\alpha}$ (where $\alpha$ is the character we have fixed in the introduction) whose restriction to the anticyclotomic tower is non-trivial. This non-trivial Euler system is the main ingredient of our results later.

\begin{theorem}
\label{thm:explicitreciprocity}$\,$\\
\textup{\textbf{(i)}} For all $n\ge 0$, the class $\textup{BF}_{D_n}$ falls in the Greenberg Selmer group $H^1_\textup{Gr}(D_n,T_{f,\alpha})$ if $\epsilon(f/K)=-1$. \\\\
\textup{\textbf{(ii)}}
Suppose either
\begin{itemize}
\item[($+$)] $N^-$ is a product of odd number of primes, $p^2\nmid N$ and $(\rm{disc}(K/\QQ),N)=1$, or that
\item[($-$)] $N^-$ is a product of even number of primes, $p\nmid N\rm{disc}(K/\QQ)$.
\end{itemize}
Then the tower $\textup{BF}_{D_\infty} \in H^1(K,\mathbb{T}_{f,\alpha}^\ac)$ of Beilinson-Flach classes along the anticyclotomic $\ZZ_p$-extension is non-trivial.

\end{theorem}
\begin{proof}
We already know that the localization of $\BF_{D_\infty}$ at ${\fp^c}$ lies inside $H^1_\textup{Gr}(K_{{\fp^c}},\mathbb{T}_{f,\alpha}^\ac)$. Hence, $\BF_{D_\infty}$ is an element of $H^1_\textup{Gr}(K,\mathbb{T}_{f,\alpha}^\ac)$ if and only if its localization at $\fp$ is contained in $H^1_\textup{Gr}(K_\fp,\mathbb{T}_{f,\alpha}^\ac)$.
 
Note that $\col^{(1)}$ is injective by \cite[Proposition~7.6.2]{LLZ2} (see also \cite[\S3]{LVZ}) and we have an isomorphism
\[
H^1(K_{\fp},T_{f,\alpha}/F^+T_{f,\alpha}\otimes\LL^{\iota})\cong\frac{ H^1(K_{\fp},\TT_{f,\alpha})}{H^1_\textup{Gr}(K_{\fp},\TT_{f,\alpha})}\,.
\]
Thanks to the explicit reciprocity law, the localization of $\BF_{D_\infty}$ at $\fp$ lies inside $H^1_\textup{Gr}(K_\fp,\mathbb{T}_{f,\alpha}^\ac)$ if and only if the anticyclotomic projection $\mathfrak{L}_{f,0}^{(\alpha)}$ of $\mathfrak{L}_{f}^{(\alpha)}$ vanishes. The first part now follows from Corollary~\ref{cor:cyclotomicderivative}.

The second part follows similarly as a consequence of the explicit reciprocity law for the Beilinson-Flach elements and the generic non-vanishing results in Theorems~\ref{thm:chidahsieh} and \ref{thm:hsiehnonvanishing}.
\end{proof}

\subsection{Anticyclotomic Beilinson-Flach Euler system} 
\label{subsec:ESargument}
 We  set $\mathfrak{X}_?(f\otimes\alpha/D_\infty):=H^1_{\FFF^*_?}(K,\TT_{f,\alpha}^{\ac,\vee}(1))^\vee$ for $?=+,-,\textup{Gr}$. We retain throughout this section our hypothesis that $N^-$ is square-free. In addition to our running hypotheses  \textup{\textbf{(H.Im.)}}, \textup{\textbf{(H.Dist.)}} and \textup{\textbf{(H.SS.)}}, the hypothesis  \textup{\textbf{(H.nEZ.)}} is also in effect until the end of Section~\ref{subsec:rubinsformula}. We remind the reader that  \textup{\textbf{(H.nEZ.)}} is redundant unless \textbf{(wt-2)} holds true.

The following is a weak form of the ($\alpha$-isotypical) anticyclotomic main conjecture for the eigenform $f$ without $p$-adic $L$-functions.
\begin{theorem}
\label{thm:mainconjwithoutpadicLfunc}
Suppose either $N^-$ is a product of odd number of primes, $p^2\nmid N$ and $(\rm{disc}(K/\QQ),N)=1$, or else that $N^-$ is a product of even number of primes and $p\nmid N\rm{disc}(K/\QQ)$.
Then the $\LL_{\ac}^{\mathfrak{o}}$-module  $\mathfrak{X}_+(f\otimes\alpha/D_{\infty})$  is torsion and $H^1_{\FFF_+}(K,\TT_{f,\alpha}^{\ac})$ is of rank one. Furthermore, 
$$\Char\left(\mathfrak{X}_+(f\otimes\alpha/D_{\infty})\right) \Big{|}\, \Char\left(H^1_{\FFF_+}(K,\TT_{f,\alpha}^{\ac})\big{/}\LL_{\ac}^{\mathfrak{o}}\cdot\textup{BF}_{D_\infty}\right)\,.$$
\end{theorem}

\begin{proof}
All this is a consequence of the locally restricted Euler system machine, as discussed in \cite[Appendix A]{kbbleiPLMS} (see also \cite[\S 12]{KLZ2}). We give a brief outline here indicating how the results therein apply. 

Let $\mathfrak{m}, \mathfrak{m}^\prime$ be two moduli as in the second paragraph of Section~\ref{subsubsec:ESoverK} with the property that $\mathfrak{m}\mid\mathfrak{m}^\prime$.  Let $\Delta(\mathfrak{m})$ denote the corresponding ray class group; similarly define $\Delta(\mathfrak{m}^\prime)$.  We set $\LL_\ac^{\mathfrak{o}}(\mathfrak{m}):=\LL_\ac^{\mathfrak{o}} \otimes_{\ZZ_p}\ZZ_p[\Delta(\mathfrak{m})]$ and  we likewise define the ring $\LL_\ac^{\mathfrak{o}}(\mathfrak{m}^\prime)$. Under our running hypothesis \textup{\textbf{(H.nEZ.)}}, it follows using \cite[\S 2.1]{kbbCMYager} that
\begin{itemize}
\item the $\LL_\ac^{\mathfrak{o}}(\mathfrak{m})$-module $H^1(K(\mathfrak{m})_{\fp^c},\TT^\ac_{f,\alpha})$ is free of rank two,
\item the $\LL_\ac^{\mathfrak{o}}(\mathfrak{m})$-submodule 
\begin{align*}
H^1_{\textup{Gr}}(K(\mathfrak{m})_{\fp^c},\TT^\ac_{f,\alpha})&:=\textup{im}\left(H^1(K(\mathfrak{m})_{\fp^c},F^+\TT^\ac_{f,\alpha})\ra H^1(K(\mathfrak{m})_{\fp^c},\TT^\ac_{f,\alpha})\right)\\
&\subset H^1(K(\mathfrak{m})_{\fp^c},\TT^\ac_{f,\alpha})
\end{align*}
is a direct summand of rank one,
\item both corestriction maps 
$$H^1(K(\mathfrak{m}^\prime)_{\fp^c},\TT^\ac_{f,\alpha})\lra H^1(K(\mathfrak{m})_{\fp^c},\TT^\ac_{f,\alpha})$$
$$H^1_\Gr(K(\mathfrak{m}^\prime)_{\fp^c},\TT^\ac_{f,\alpha})\lra H^1_\Gr(K(\mathfrak{m})_{\fp^c},\TT^\ac_{f,\alpha})$$
are surjective.
\end{itemize}
The $L$-restricted Selmer structure $\FFF_L$ considered in \cite[Definition A.11]{kbbleiPLMS} corresponds to our $\FFF_+$. The $\fp^c$-local properties of the Beilinson-Flach elements (that they are Greenberg at $\fp^c$; which amounts to saying in the terminology of op.cit. that the Beilinson-Flach element Euler system is an $L$-restricted Euler system) used together with Theorem A.11 of op.cit. give rise to a Kolyvagin system for the Selmer structure  $\FFF_+$. The divisibility in the statement of our theorem follows as a consequence of this fact (relying on the the fact that the Selmer structure $\FFF_+$ is Cartesian; see also Theorem 12.3.4 and Corollary 12.3.5 of \cite{KLZ2}). If we additionally use Theorem~\ref{thm:explicitreciprocity}, the claim that $\mathfrak{X}_+(f\otimes\alpha/D_{\infty})$ is torsion also follows. The assertion that the $\LL_\ac^\mathfrak{o}$-module  $H^1_{\FFF_+}(K,\TT_{f,\alpha}^{\ac})$ is of rank one now follows from Poitou-Tate global duality.
\end{proof}
\begin{theorem}
\label{thm:mainconjtakeone}$\,$\\
\textup{\bf{(i)}} Suppose $N^-$ is a product of odd number of primes, $p^2\nmid N$ and $(\rm{disc}(K/\QQ),N)=1$ (so that the condition (\textup{Sign +}) holds true). Then the $\LL_{\ac}^{\mathfrak{o}}$-module  $\mathfrak{X}(f\otimes\alpha/D_\infty)$  is torsion and 
$$\Char\left(\mathfrak{X}(f\otimes\alpha/D_\infty)\right) \Big{|}\, \Char\left(H^1_{\plusf}(K,\TT_{f,\alpha}^{\ac})\Big{/}\LL_{\ac}^{\mathfrak{o}}\cdot\textup{res}_{+/\f}\,(\textup{BF}_{D_\infty})\right)\,.$$
Equality (resp., equality up to powers of $\varpi$) holds if and only if the divisibility in Theorem~\ref{thm:mainconjwithoutpadicLfunc} is an equality (resp., an equality up to powers of $\varpi$). 
\\\\
\textup{\bf{(ii)}} Suppose that $N^-$ is a product of even number of primes and $p\nmid N\rm{disc}(K/\QQ)$ (so that the condition (\textup{Sign }$-$) holds true). Then $H^1_{\FFF_-}(K,\TT_{f,\alpha}^{\ac})=0$ and all the three $\LL_{\ac}^{\mathfrak{o}}$-modules  $H^1_{\FFF_\Gr}(K,\TT_{f,\alpha}^{\ac})$, $\mathfrak{X}(f\otimes\alpha/D_\infty)$ and $\mathfrak{X}_-(f\otimes\alpha/D_\infty)$ are of rank one.
\end{theorem}
\begin{proof}
To prove (i), note that we have $\LL_\ac^\mathfrak{o}\cdot\textup{BF}_{D_\infty}\cap H^1_{\FFF_\Gr}(K,\TT_{f,\alpha}^\ac)=0$ by Theorem~\ref{thm:explicitreciprocity}. It follows from Theorem~\ref{thm:mainconjwithoutpadicLfunc}  that the $\LL_\ac^\mathfrak{o}$-module $H^1_{\FFF_\Gr}(K,\TT_{f,\alpha}^\ac)$ is torsion. Since we assumed \textup{\textbf{(H.Im.)}}, we conclude that $H^1_{\FFF_\Gr}(K,\TT_{f,\alpha}^\ac)=0$. Both claims in (i) now follow from Poitou-Tate global duality.
 
We now prove (ii). In this set up, Theorem~\ref{thm:explicitreciprocity} shows that the $\LL_{\ac}^{\mathfrak{o}}$-module $H^1_{\FFF_\Gr}(K,\TT_{f,\alpha}^{\ac})$ has positive rank. On the other hand, the containment 
$$H^1_{\FFF_\Gr}(K,\TT_{f,\alpha}^{\ac}) \subset H^1_{\FFF_+}(K,\TT_{f,\alpha}^{\ac})$$ 
together with Theorem~\ref{thm:mainconjwithoutpadicLfunc} shows that $H^1_{\FFF_\Gr}(K,\TT_{f,\alpha}^{\ac})$ is a $\LL_{\ac}^{\mathfrak{o}}$-module of rank one. Furthermore, noticing also that 
$$H^1_{\FFF_-}(K,\TT_{f,\alpha}^{\ac})=\ker\left(H^1_{\FFF_\Gr}(K,\TT_{f,\alpha}^{\ac})\stackrel{\res_{\textup{f}/-}}{\lra} H^1_{\FFF_{\Gr}}(K_\fp,\TT_{f,\alpha}^{\ac})\right)$$
and relying on the fact that $\LL_{\ac}^{\mathfrak{o}}$-module $H^1_{\FFF_{\Gr}}(K_\fp,\TT_{f,\alpha}^{\ac})$ has rank one, we conclude that $H^1_{\FFF_-}(K,\TT_{f,\alpha}^{\ac})=0$, as required. 

The remaining two assertions in (ii) are consequences of well-known global duality statements. Choose an auxiliary character $\chi:\Gamma^\ac\ra \ooo^\times$ (where $\ooo$ is finite flat over $\ZZ_p$) of finite order with the properties that 
\begin{itemize}
\item the prime $\pi_\chi:=\gamma_\ac-\chi(\gamma^{-1}_\ac)\in \LL_{\ac}^\ooo$ 
does not divide the characteristic ideal of the $\LL_{\ac}^\ooo$-module $\mathfrak{X}(f\otimes\alpha/D_\infty)_{\textup{tor}}^\iota \otimes\ooo$\,;
\item $\pi_\chi$ does not divide the characteristic ideal of $H^2(K_p,\TT_{f,\alpha}^\ac)$\,.
\end{itemize}
Note that almost all characters of $\Gamma^\ac$ of finite order satisfies this requirement. For a $\LL_{\ac}$-module $X$, write $X_\ooo$ in place the $\LL_{\ac}^\ooo$-module $X\otimes\LL_{\ac}^\ooo$. Comparing the classical Selmer groups with Nekov\'a\v{r}'s extended Selmer groups and using his control theorem \cite[Corollary 8.10.2]{nekovar06}, we infer that 
\begin{equation}\label{eqn:nekcontrol}
H^1_{\FFF_{\Gr}}(K,\TT_{f,\alpha}^{\ac})_\ooo/\pi_\chi\hookrightarrow H^1_{\FFF_{\Gr}}(K,T_{f,\alpha}(\chi^{-1}))\twoheadrightarrow \widetilde{H}^2_{\textup{f, Iw}}(D_\infty/K,T_{f,\alpha})_\ooo[\pi_\chi]
\end{equation}
where $\widetilde{H}^2_{\textup{f, Iw}}(D_\infty/K,T_{f,\alpha})$ is the degree two cohomology of an appropriately defined (by using the filtration $F^+V_{f,\alpha}\subset V_{f,\alpha}$ to define local conditions) Selmer complex.  The module $\widetilde{H}^2_{\textup{f, Iw}}(D_\infty/K,T_{f,\alpha})$ is related to $\mathfrak{X}(f\otimes\alpha/D_\infty)$ via the global duality statements~\cite[(8.9.6.2]{nekovar06} by the isomorphism 
$$\mathfrak{X}(f\otimes\alpha/D_\infty) \cong \widetilde{H}^2_{\textup{f, Iw}}(D_\infty/K,T_{f,\alpha})^\iota\,.$$
Our choice of $\chi$ shows that $\widetilde{H}^2_{\textup{f, Iw}}(D_\infty/K,T_{f,\alpha})_\ooo[\pi_\chi]$ is finite and we conclude using (\ref{eqn:nekcontrol}) that 
\begin{equation}\label{eqn:greenberggenericallyrankone}
\textup{rank}_\ooo H^1_{\FFF_{\Gr}}(K,T_{f,\alpha}(\chi^{-1})) = \textup{rank}_{\LL_{\ac}^\mathfrak{o}} H^1_{\FFF_{\Gr}}(K,\TT_{f,\alpha}^\ac)=1\,.
\end{equation}
Let us denote the Greenberg Selmer structure on $T_{f,\alpha}(\chi^{-1})$ (given by the filtration $F^+T_{f,\alpha}(\chi^{-1})\subset T_{f,\alpha}(\chi^{-1})$) by $\FFF_{\Gr,\chi}$. Note thanks to our choice of $\chi$ that $H^1_{\FFF_{\Gr}}(K_p,T_{f,\alpha}(\chi^{-1}))$ is contained in $H^1_{\FFF_{\Gr,\chi}}(K_p,T_{f,\alpha}(\chi^{-1}))$ with finite index and therefore, the Selmer group $H^1_{\FFF_{\Gr}}(K,T_{f,\alpha}(\chi^{-1}))$ is also contained in the Selmer group $H^1_{\FFF_{\Gr},\chi}(K_p,T_{f,\alpha}(\chi^{-1}))$ with finite index. Furthermore, it is easy to see that the core Selmer rank of the Selmer structure $\FFF_{\Gr,\chi}$ (in the sense of \cite{mr02}) on $T_{f,\alpha}(\chi^{-1})$ equals $0$. This in turn shows
\begin{align*}\textup{rank}_\ooo\left(\mathfrak{X}(f\otimes\alpha/D_\infty)\big{/}\pi_{\chi^{-1}}\mathfrak{X}(f\otimes\alpha/D_\infty)\right)&=\textup{corank}_\ooo \,H^1_{\FFF_{\Gr}^*}(K,T_{f,\alpha}^\vee(1+\chi))\\
&=\textup{corank}_\ooo \,H^1_{\FFF_{\Gr,\chi}^*}(K,T_{f,\alpha}^\vee(1+\chi))\\
&=\textup{rank}_\ooo \,H^1_{\FFF_{\Gr,\chi}}(K,T_{f,\alpha}(\chi^{-1}))=1
\end{align*}
where the first equality follows from a classical control theorem for Greenberg Selmer groups and $M(1+\chi)$ is a short hand for the twist $M(1)(\chi)$, the second thanks to our choice of the character $\chi$, the third using \cite[Corollary 5.2.6]{mr02} and the last equally from (\ref{eqn:greenberggenericallyrankone}). Allowing $\chi$ vary, we conclude that $\mathfrak{X}(f\otimes\alpha/D_\infty)$ is of rank one. 

The proof that $\mathfrak{X}_-(f\otimes\alpha/D_\infty)$ is free of rank one follows from the global duality sequence
\begin{align*}0\lra H^1_{\FFF_{-}}(K,\TT_{f,\alpha}^{\ac})\lra H^1_{\FFF_{\Gr}}(K,&\TT_{f,\alpha}^{\ac})\lra H^1_{\FFF_{\Gr}}(K_{\fp^c},\TT_{f,\alpha}^{\ac})\\
&\lra \mathfrak{X}_-(f\otimes\alpha/D_\infty)\lra \mathfrak{X}(f\otimes\alpha/D_\infty)\lra 0\,.
\end{align*}
and the previous portions of the theorem.
\end{proof}
\begin{proposition}
\label{prop:torsionsubmodulepofXminus}
Under the hypotheses of Theorem~\ref{thm:mainconjtakeone}(ii) we have 
$$\Char\left(\mathfrak{X}_+(f\otimes\alpha/D_\infty)\right) =\Char\left(\mathfrak{X}_-(f\otimes\alpha/D_\infty)_{\textup{tor}}\right)\, $$
up to a power of $\varpi$.
\end{proposition}
\begin{proof}
The structure of the proof of this proposition is similar to the proof of the last portion of Theorem~\ref{thm:mainconjtakeone} and we follow the notation we have introduced therein (except that we no longer assume that $\chi$ has finite order). We define the Selmer structure $\FFF_{+,\chi}$ (resp., $\FFF_{-,\chi}$) on $T_{f,\alpha}(\chi^{-1})$ by relaxing $\FFF_{\Gr,\chi}$ we have introduced in the proof of Theorem~\ref{thm:mainconjtakeone} at $\fp$ (resp., by replacing the local conditions at $\fp$ by the strict conditions). 

Since we assumed when $k=2$ that $p\nmid N$ and $p>5$, we have $H^0(K_p,F^-\overline{T}_{f,\alpha})=0$ and in turn also that the module $H^1(K_p,F^-T_{f,\alpha}(\chi^{-1}))$ is torsion free. Thence, it follows from \cite[Lemma 3.7.1]{mr02} that both Selmer structures $\FFF_{+,\chi}$ and $\FFF_{-,\chi}$ are cartesian in the sense of \cite[Definition 1.1.4]{mr02}. This shows that the two local conditions on $T_{f,\alpha}^\vee(1+\chi)[ \varpi_\Phi^m]\cong T_{f,\alpha^c}(\chi)/\varpi^m_\Phi T_{f,\alpha^c}(\chi)$ propagated from the two Selmer structures $\FFF_{+,\chi}^*$ (resp., $\FFF_{-,\chi}^*$) and $\FFF_{-,\chi}$ (resp., $\FFF_{+,\chi}$) coincide for every positive integer $m$. Applying Lemma 3.5.3 and Theorem 5.2.5 of \cite{mr02}, we have
$$H^1_{\FFF_{-,\chi}^*}(K,T_{f,\alpha}^\vee(1+\chi))[\varpi^m_\Phi]\cong \oo/\varpi_\Phi^m\oplus H^1_{\FFF_{+,\chi}^*}(K,T_{f,\alpha}^\vee(1+\chi))[\varpi^m_\Phi]\,$$
in which passing to direct limit as $n\ra\infty$ and taking Pontryagin duals, we conclude that
\begin{equation}\label{eqn:cmopareminustoplus}
H^1_{\FFF_{-,\chi}^*}(K,T_{f,\alpha}^\vee(1+\chi))^\vee\cong \ooo\oplus H^1_{\FFF_{+,\chi}^*}(K,T_{f,\alpha}^\vee(1+\chi))^\vee\,.
\end{equation}
On the other hand, by a classical control theorem (applied with the two Selmer structures $\FFF_\pm$) we have a surjection
$$\mathfrak{X}_{\pm}(f\otimes\alpha/D_{\infty})\big{/}\pi_{\chi^{-1}}\mathfrak{X}_{\pm}(f\otimes\alpha/D_{\infty})\twoheadrightarrow H^1_{\FFF_{\pm,\chi}^*}(K,T_{f,\alpha}^\vee(1+\chi))^\vee$$
which has finite kernel whose order is bounded independently of $\chi$ as $\chi$ ranges on our chosen set of anticyclotomic characters. This combined with (\ref{eqn:cmopareminustoplus}) shows that
$$\textup{length}_\ooo\,\left( \mathfrak{X}_{-}(f\otimes\alpha/D_{\infty})_{\textup{tor}}\Big{/}\pi_{\chi^{-1}}\right)-\textup{length}_\ooo\,\left(\mathfrak{X}_+(f\otimes\alpha/D_\infty)\Big{/}\pi_{\chi^{-1}}\right) $$
is bounded independently of $\chi$. It is not hard to see that our proposition follows from this fact, c.f. the second portion of the proof of \cite[Lemma 1.2.6]{agboolahowardordinary}.
\end{proof}
\begin{defn}
\label{def:minusregulator}
In the situation of Theorem~\ref{thm:mainconjtakeone}(ii), we define the ideal
\begin{align*}\mathfrak{l}_-&:=\Char\left(H^1_{\f/-}(K_p,\TT^\ac_{f,\alpha})\Big{/}\res_{\f/-}\left(H^1_{\FFgr}(K,\TT_{f,\alpha}^\ac)\right)\right) \\
&=\Char\left(H^1_{\FFgr}(K_{\fp^c},\TT_{f,\alpha}^\ac)\Big{/}\res_{\fp^c}\left(H^1_{\FFgr}(K,\TT_{f,\alpha}^\ac)\right)\right) 
\end{align*}
where we recall that $\res_{\f/-}: H^1_{\FFgr}(K,\TT^\ac_{f,\alpha})\ra H^1_{\f/-}(K_p,\TT^\ac_{f,\alpha})$ is the natural map (which we may identify with the map $\res_{\fp^c}: H^1_{\FFgr}(K,\TT^\ac_{f,\alpha})\ra H^1_{\FFgr}(K_{\fp^c},\TT^\ac_{f,\alpha})$).
\end{defn}
\begin{lemma}
\label{lem:PTdualityandtorsionsubmodules}
Under the hypotheses of Theorem~\ref{thm:mainconjtakeone} we have
$$\Char\left(\mathfrak{X}_-(f\otimes\alpha/D_\infty)_{\textup{tor}}\right)=\Char\left(\mathfrak{X}(f\otimes\alpha/D_\infty)_{\textup{tor}}\right)\cdot \mathfrak{l}_-\,.$$
\end{lemma}

\begin{proof}
This follows by taking $\LL_{\ac}^{\mathfrak{o}}$-torsion on the global duality sequence 
$$0\ra H^1_{\FFgr}(K,\TT^\ac_{f,\alpha})\ra H^1_{\f/-}(K_p,\TT^\ac_{f,\alpha})\ra\mathfrak{X}_-(f\otimes\alpha/D_\infty)\ra \mathfrak{X}(f\otimes\alpha/D_\infty)\ra0\,,$$
where the exactness on the left follows from Theorem~\ref{thm:mainconjtakeone}(ii).
\end{proof}
\subsection{Two-variable main conjecture and upgrading divisibilities}
\label{subsec:twovarmainconj}
Our goal in this section is to prove that the divisibility in Theorem~\ref{thm:mainconjwithoutpadicLfunc} may be improved to an equality up to powers of $\varpi$, by making use of the main results of \cite{skinnerurbanmainconj, xinwanwanrankinselberg} whenever they apply.

We first state the $2$-variable version of Theorem~\ref{thm:mainconjwithoutpadicLfunc}. Recall the $\LL$-modules $\mathfrak{X}_+(f\otimes\alpha/K_\infty)$ and $\mathfrak{X}_{\Gr}(f\otimes\alpha/K_\infty)$ which are defined as the Pontryagin duals of the two Selmer groups $H^1_{\FFF_+^*}(K,\TT_{f,\alpha}^\vee(1))$ and $H^1_{\FFF_\Gr^*}(K,\TT_{f,\alpha}^\vee(1))$. In this portion of our article, we will consider another Selmer structure $\FFF_{+-}$, which is obtained from $\FFF_+$ by altering local conditions only at $\fp^c$ and imposing strict local conditions (in place of Greenberg local conditions). We shall write $\mathfrak{X}_{+-}(f\otimes\alpha/K_\infty)$ for the Pontryagin dual of the dual Selmer group $H^1_{\FFF_{+-}^*}(K,\TT_{f,\alpha}^\vee(1))$. We also recall  the $p$-adic $L$-functions $\mathfrak{L}_f^{(\alpha)}$ and $\mathfrak{L}_{p,\Sigma^{(2)}}^{(\alpha)}$. 
\begin{theorem}
\label{thm:2varmainconjwithoutpadicL}
\textup{\textbf{(i)}}
The $\LL^{\mathfrak{o}}$-module  $\mathfrak{X}_+(f\otimes\alpha/K_{\infty})$  is torsion and $H^1_{\FFF_+}(K,\TT_{f,\alpha})$ is of rank one. Furthermore, 
\begin{equation}
\label{eqn:KLZplusmainconjintwovar}
\Char_{\LL^\mathfrak{o}}\left(\mathfrak{X}_+(f\otimes\alpha/K_{\infty})\right) \Big{|}\, \Char\left(H^1_{\FFF_+}(K,\TT_{f,\alpha})\big{/}\LL^\mathfrak{o}\cdot\textup{BF}_{K_\infty}\right)\,,
\end{equation}
\begin{equation}
\label{eqn:KLZgreenbergmainconjintwovar}
\Char_{\LL^\mathfrak{o}}\left(\mathfrak{X}_\Gr(f\otimes\alpha/K_{\infty})\right)\otimes \RR_L\,\, \Big{|}\,\, \mathfrak{L}_f^{(\alpha)}\cdot\RR_L\,,
\end{equation}
\begin{equation}
\label{eqn:KLZplusminusmainconjintwovar}
\Char_{\LL^\mathfrak{o}}\left(\mathfrak{X}_{+-}(f\otimes\alpha/K_{\infty})\right) \,\,\Big{|}\,\, \Char_{\LL^\mathfrak{o}}\left(H^1_{\FFF_{\Gr}}(K_{\fp^c},\TT_{f,\alpha})\big{/}{\LL^\mathfrak{o}}\cdot\res_{\fp^c}\left(\textup{BF}_{K_\infty}\right)\right)\,. 
\end{equation}
We have equality (resp., equality up to powers of $\varpi$) in one of (\ref{eqn:KLZplusmainconjintwovar}), (\ref{eqn:KLZgreenbergmainconjintwovar}) or (\ref{eqn:KLZplusminusmainconjintwovar}) if and only if we have in all of them.
\\\\\textup{\textbf{(ii)}} The ${\LL^\mathfrak{o}}$-module $\mathfrak{X}_{+-}(f\otimes\alpha/K_{\infty})$ is torsion. If we suppose in addition that there is no prime $v\mid \mathfrak{f}$ such that $v^c \mid \mathfrak{f}$, then $\mathfrak{X}_\Gr(f\otimes\alpha/K_{\infty})$ is torsion as well.

\end{theorem}
\begin{proof}
We first explain that the $p$-adic $L$-function $\mathfrak{L}_{f,\Sigma^{(2)}}^{(\alpha)}$ is non-trivial, which will follow from its interpolation property. The argument we present here is borrowed from the proof of~\cite[Theorem 9.2.1]{KLZ0} (adapted to our set up so as to be consistent with our normalization of its interpolation property that is shifted from $s=1$ to $s=k/2$). We recall that  $\mathfrak{L}_{f,\Sigma^{(2)}}^{(\alpha)}$ interpolates the critical values $L(f/K, \psi, k/2)$ for Hecke characters $\psi$ of conductor dividing $\mathfrak{f}p^\infty$ and infinity-type $(a,b)$ with $a\leq -k/2$ and $b \geq k/2$. The Euler product determining the Hecke $L$-function $L(f/K, \psi, s)$ converges at $s = k/2$ if we consider $\psi$ of infinity-type $(a, b)$ with $a+b > 1$. Thence the value $L(f/K, \psi, 1)$ cannot be zero for such $\psi$. Since a Hecke character $\psi$ with these properties exists, the non-triviality of $\mathfrak{L}_{f,\Sigma^{(2)}}^{(\alpha)}$ follows.

Combining this fact with the second explicit reciprocity law (\ref{eqn:secondreciprocitylaw}), we conclude that the class $\textup{BF}_{K_\infty}$ is non-trivial. The first two assertions in (i) as well as (\ref{eqn:KLZplusmainconjintwovar}) follows as a consequence of the locally restricted Euler system machinery (as we have utilized in the proof of Theorem~\ref{thm:mainconjwithoutpadicLfunc} for the two-dimensional ring $\LL^\mathfrak{o}_\ac$; Ochiai's dimension reduction argument based on \cite[Proposition 3.6]{ochiai05} allows us to treat the three-dimensional coefficient ring $\LL^\mathfrak{o}$ as well). The rest of (i) follows employing Poitou-Tate global duality (along with the first explicit reciprocity law (\ref{eqn:firstreciprocitylaw}) to deduce (\ref{eqn:KLZgreenbergmainconjintwovar})).  
 
We now prove (ii). Our first assertion therein follows at once once from (\ref{eqn:KLZplusminusmainconjintwovar}), the second explicit reciprocity law (\ref{eqn:secondreciprocitylaw}) and the non-triviality of $\mathfrak{L}_{f,\Sigma^{(2)}}^{(\alpha)}$ that we verified above. The second assertion will follow from (\ref{eqn:KLZgreenbergmainconjintwovar}) once we verify that the $p$-adic $L$-function $\mathfrak{L}_f^{(\alpha)}$ is non-trivial. This is what we carry out now. The $p$-adic $L$-function $\mathfrak{L}_f^{(\alpha)}$ interpolates the critical values $L(f/K, \psi, k/2)$ for Hecke characters $\psi$ of conductor dividing $\mathfrak{f}p^\infty$ and infinity-type $(a,b)$ with $1-k/2\leq a,b\leq  k/2-1$. If $k\geq 4$, then this interval contains a pair with $a+b>1$, so that we can argue as in the first paragraph of this proof. If $k=2$, it follows from \cite[Theorem 2]{rohrlich88} that there are infinitely many ray class characters $\psi$ unramified outside $\mathfrak{fp}$ with $L(f/K,\psi,1)\neq 0$. The proof that $\mathfrak{L}_f^{(\alpha)}\neq 0$ follows.
\end{proof}

\begin{theorem}[Kato, Skinner-Urban, Wan]
\label{thm:skinnerurbanwan}
Suppose that $N=N^+N^-$ where $N^+$ is only divisible by primes that are split in $K/\QQ$ and $N^-$ with those which are inert in $K/\QQ$. Assume that $N^-$ is square-free and has an odd number of prime factors (so that the condition (\textup{Sign +}) holds true). Then, the divisibility in (\ref{eqn:KLZgreenbergmainconjintwovar}) is in fact an equality.
\end{theorem}
\begin{proof}
The only point here is the comparison of Hida's two-variable $p$-adic $L$-function $\mathfrak{L}_f^{(\alpha)}$ to the Skinner-Urban $p$-adic $L$-function $\LSU$ (that had made its appearance in Corolary~\ref{cor:SU}). We let 
$$\textup{Tw}_{k/2-1}:\, \mathcal{R}_L\stackrel{\sim}{\lra} \mathcal{R}_L$$ 
denote the twisting isomorphism induced by $\gamma \mapsto \chi_\cyc^{k/2-1}(\gamma)\gamma$ for $\gamma \in \Gamma$. It follows from the two-variable main conjecture proved by Skinner-Urban \cite{skinnerurbanmainconj} in the definite case (as improved by~\cite{xinwanwanhilbert} to its current form and twisted appropriately) used together with our divisibility result (\ref{eqn:KLZgreenbergmainconjintwovar}) that $\textup{Tw}_{k/2-1}\left(\LSU\right)$ divides $\mathfrak{L}_f^{(\alpha)}$. We conclude by Corollary~\ref{cor:SU} that $\textup{Tw}_{k/2-1}\left(\LSU\right)$ and $\mathfrak{L}_f^{(\alpha)}$ generate the same ideal of $\mathcal{R}_L$ and it follows that our version of the two-variable main conjecture is equivalent to that proved by Skinner-Urban and Wan.
\end{proof}

Making use of a standard control argument, we may descend to the anticyclotomic line and deduce:
\begin{corollary}
\label{cor:mainanticyclosharpdefinite}
Suppose that there is no prime $v\mid \mathfrak{f}$ such that $v^c \mid \mathfrak{f}$. Then under the assumptions of  Theorem~\ref{thm:skinnerurbanwan}, the divisibility statement in Theorem~\ref{thm:mainconjtakeone}(i) is in fact an equality up to powers of $\varpi$.
\end{corollary}
\begin{proof}
This  follows from Theorem~\ref{thm:skinnerurbanwan} thanks to Corollary~\ref{cor:SU} after a standard Iwasawa theoretic descent argument. We shall still include a detailed proof here for two reasons: First, for the sake of completeness, as we were unable to find it elsewhere in the literature. Second, the proof we shall provide here will be equally useful in the indefinite set up.

We will provide details for the proof that we have equality in Theorem~\ref{thm:mainconjwithoutpadicLfunc}. Although we could have directly prove this statement more directly (namely, without going through the Selmer structure $\FFF_+$), we choose to proceed this way as our arguments work also in the indefinite case to deduce Corollary~\ref{cor:mainconjequalityindefinitecase} below. 

To do so, we shall rely on Nekov\'a\v{r}'s control theorems instead of their more classical counterparts. Consider the Selmer complexes $\widetilde{C}^{\bullet}_{\textup{f},\Iw}(K_\infty/K,T_{f,\alpha}; \Delta_\Gr)$ and $\widetilde{C}^{\bullet}_{\textup{f},\Iw}(K_\infty/K,T_{f,\alpha}; \Delta_+)$ as well as the corresponding objects $\widetilde{\mathbf{R}\Gamma}_{\textup{f},\Iw}(K_\infty/K,T_{f,\alpha}; \Delta_+)$ and $\widetilde{\mathbf{R}\Gamma}_{\textup{f},\Iw}(K_\infty/K,T_{f,\alpha}; \Delta_+)$ in the derived category. Here, the local conditions $\Delta_\Gr$ are given by the filtration $F^{+}T_{f,\alpha}\subset T_{f,\alpha}$ at both primes $\fp$ and $\fp^c$ above $p$, whereas the local condition $
\Delta_+$ is given by the same filtration at $\fp^c$, where as it is given by the tautological filtration $T_{f,\alpha}\subset T_{f,\alpha}$ at $\fp$. Their cohomology $\widetilde{H}^1_{\textup{f},\Iw}(K_\infty/K,T_{f,\alpha},\Delta_\Gr)$ and $\widetilde{H}^1_{\textup{f},\Iw}(K_\infty/K,T_{f,\alpha},\Delta_+)$  in degree one agree (using \cite[Lemma 9.6.3]{nekovar06} and passing to limit) with $H^1_{\FFF_\Gr}(K,\TT_{f,\alpha})$ and $H^1_{\FFF_+}(K,\TT_{f,\alpha})$, respectively. Furthermore, it follows from Nekov\'a\v{r}'s duality \cite[8.9.6.2]{nekovar06} that we have canonical isomorphisms 
$$\widetilde{H}^2_{\textup{f},\Iw}(K_\infty/K,T_{f,\alpha},\Delta_\Gr)^\iota\stackrel{\sim}{\lra} \widetilde{H}^1_{\textup{f}}(K_S/K_\infty,T_{f,\alpha}^\vee(1),\Delta_\Gr)^\vee\,,$$
$$\widetilde{H}^2_{\textup{f},\Iw}(K_\infty/K,T_{f,\alpha},\Delta_+)^\iota\stackrel{\sim}{\lra} \widetilde{H}^1_{\textup{f}}(K_S/K_\infty,T_{f,\alpha}^\vee(1),\Delta_-)^\vee$$
where the local condition $\Delta_-$ is given by the filtration $F^{+}T_{f,\alpha}\subset T_{f,\alpha}$ at $\fp^c$, whereas it is given by the  filtration $\{0\} \subset T_{f,\alpha}$ at $\fp$. Since we assumed $p>5$ when $k=2$, we see that $H^0(K_{p},F^-\overline{T}_{f,\alpha})=0=H^0(K_{p},\overline{T}_{f,\alpha})$. It then follows by (a slight generalization of) \cite[Proposition 9.6.6(iii)]{nekovar06} that we have canonical isomorphisms 
$$\widetilde{H}^1_{\textup{f}}(K_S/K_\infty,T_{f,\alpha}^\vee(1),\Delta_-)^\vee\stackrel{\sim}{\lra}\mathfrak{X}_+(f\otimes\alpha/K_\infty)$$
$$\widetilde{H}^1_{\textup{f}}(K_S/K_\infty,T_{f,\alpha}^\vee(1),\Delta_\Gr)^\vee\stackrel{\sim}{\lra}\mathfrak{X}_{\Gr}(f\otimes\alpha/K_\infty)$$
All the constructions and discussion carries over on replacing $K_\infty$ with $D_\infty$. Let $\pi^{\ac}:\LL\ra \LL_{\ac}$ denote the augmentation map induced from $\gamma_\cyc \mapsto 1$.

In view of these identifications, Theorem~\ref{thm:skinnerurbanwan} states that
\begin{equation}
\label{eqn:mainconjwithoutpadicLrestatement}
 \Char_\LL\left(\mathfrak{H}_+^{2}(K_\infty)\right) \,\dot{=}\, \Char\left(\mathfrak{H}_+^{1}(K_\infty)\right)
\end{equation}
where ``$\dot{=}$" stands for equality up to powers of $\varpi$ and $\mathfrak{H}_+^{1}(K_\infty)$ (resp., $\mathfrak{H}_+^{2}(K_\infty)$) is a shorthand for the quotient ${\widetilde{H}^1_{\textup{f},\Iw}(K_\infty/K,T_{f,\alpha},\Delta_+)}\big{/}{\LL^{\mathfrak{o}}\cdot\textup{BF}_{K_\infty}}\,$ (resp., for  the $\LL$-module $\widetilde{H}^2_{\textup{f},\Iw}(K_\infty/K,T_f,\Delta_+)^\iota$). We also define similarly $\mathfrak{H}_+^{1}(D_\infty)$ and $\mathfrak{H}_+^{2}(D_\infty)$. Nekov\'a\v{r}'s control theorem \cite[Proposition 8.10.10]{nekovar06} yields an exact sequence
\begin{equation}
\label{eqn:nekcontrol1}
0\lra\pi^\ac\left(\mathfrak{H}_+^{1}(K_\infty)\right)\lra\, \mathfrak{H}_+^{1}(D_\infty)\lra \mathfrak{H}_+^{2}(K_\infty)[\gamma_\cyc-1]  \lra 0
\end{equation}
and an isomorphism
\begin{equation}
\label{eqn:nekcontrol2}
\pi^{\ac}\left(\mathfrak{H}_+^{2}(K_\infty)\right)\stackrel{\sim}{\lra} \mathfrak{H}_+^{2}(D_\infty)\end{equation}
(where the surjection in (\ref{eqn:nekcontrol2}) follows from the vanishing of the cohomology groups $\widetilde{H}^0_{\textup{f}}(G_{K,S},\overline{T}_{f,\alpha}^*)\cong H^0(G_{K,S},T_{f,\alpha}^*)$ in degree zero and Nekov\'a\v{r}'s duality). It follows from (\ref{eqn:nekcontrol2}) and Theorem~\ref{thm:mainconjwithoutpadicLfunc} that $\gamma_\cyc-1$ does not divide $\Char_{\LL^\mathfrak{o}}\left(\mathfrak{H}_+^{1}(K_\infty)\right)=\Char_{\LL^{\mathfrak{o}}}\left(\mathfrak{H}_+^{2}(K_\infty)\right).$ In particular, it follows from  \cite[Lemme 4 of Section 1.3]{pr84} that  
\begin{equation}
\label{eqn:nonvanishingoferrterm1}
\Char\left(\mathfrak{H}_+^{2}(K_\infty)[\gamma_\cyc-1]\right) \neq 0\,.
\end{equation}
Furthermore, it also follows from the same Lemma of \cite{pr84} that $\mathfrak{H}_+^{1}(K_\infty)[\gamma_\cyc-1]$ is a pseudo-null $\LL$-module. On the other hand, notice that the $\LL^\mathfrak{o}$-module ${\widetilde{H}^1_{\textup{f},\Iw}(K_\infty/K,T_{f,\alpha},\Delta_+)}$ is torsion-free (thanks to \textup{\textbf{(H.Im.)}}) and it contains the module ${\LL^{\mathfrak{o}}\cdot\textup{BF}_{K_\infty}}$ as a free-submodule of rank one. It follows from \cite[Lemma 6.5]{pollackrubin04} that their quotient module $\mathfrak{H}_+^{1}(K_\infty)$ does not have any $\LL$-pseudo-null submodules. This in turn shows that
\begin{equation}
\label{eqn:nonvanishingoferrterm2}
\mathfrak{H}_+^{1}(K_\infty)[\gamma_\cyc-1]= 0\,.
\end{equation}
We now have,
\begin{align}
\label{alignedequalities1}
\Char\left(\mathfrak{H}_+^{2}(D_\infty)\right)&\,{=}\,\pi^{\ac}\left(\Char_\LL\left(\mathfrak{H}_+^{2}(K_\infty)\right)\right)\cdot \Char\left(\mathfrak{H}_+^{2}(K_\infty)[\gamma_\cyc-1]\right)
\\
\label{alignedequalities2}&\,\dot{=}\,\pi^{\ac}\left(\Char_\LL\left(\mathfrak{H}_+^{1}(K_\infty)\right)\right)\cdot \Char\left(\mathfrak{H}_+^{2}(K_\infty)[\gamma_\cyc-1]\right)
\\
\label{alignedequalities3}&\,{=}\,\Char\left(\pi^{\ac}\left(\mathfrak{H}_+^{1}(K_\infty)\right)\right)\cdot\Char\left(\mathfrak{H}_+^{2}(K_\infty)[\gamma_\cyc-1]\right)\\
\label{alignedequalities4}&\,{=}\,\Char\left(\mathfrak{H}_+^{1}(K_\infty)\right)
\end{align}
where (\ref{alignedequalities1}) follows from (\ref{eqn:nekcontrol2}) and \cite[Lemme 4 of Section 1.3]{pr84}; (\ref{alignedequalities2}) from (\ref{eqn:mainconjwithoutpadicLrestatement}); (\ref{alignedequalities3}) from  \cite[Lemme 4 of Section 1.3]{pr84} and (\ref{eqn:nonvanishingoferrterm2}) and finally (\ref{alignedequalities4}) from (\ref{eqn:nekcontrol1}). We now proved the identity  
$$\Char\left(\mathfrak{H}_+^{2}(D_\infty)\right)\,\dot{=}\, \Char\left(\mathfrak{H}_+^{1}(D_\infty)\right),$$
which is a restatement of the asserted equality up to powers of $\varpi$.
\end{proof}
\begin{remark}
Under a mild additional hypothesis on the Tamagawa factors associated to $T_f$  at primes $v \nmid p$ by Fontaine and Perrin-Riou, one may give an alternative and somewhat more conceptual proof of Corollary~\ref{cor:mainanticyclosharpdefinite} using the structure of $\LL$-adic Kolyvagin systems. We believe that the argument might be of independent interest (as it by-passes the need for the delicate information on the pseudo-null submodules) and present it in this remark. 

The main theorems of \cite{kbbdeform} (see also \cite[Appendix A]{kbbCMabvar} and \cite[Appendix B]{kbbleiPLMS}) show that the module $\textup{\textbf{KS}}(\FFF_+,\TT_{f,\alpha})$ of $\LL^\mathfrak{o}$-adic Kolyvagin systems for the Selmer structure $\FFF_+$ on $\TT_{f,\alpha}$ is a free $\LL^\mathfrak{o}$-module of rank one. Let $\pmb{\kappa} \in \textup{\textbf{KS}}(\FFF_+,\TT_{f,\alpha})$ be a generator and $\kappa_1 \in H^1(K,\TT_{f,\alpha})$ be its initial class. We also let $\pmb{\kappa}^{\BF}_{K_\infty}$ denote the Beilinson-Flach Kolyvagin system (that descends from the Beilinson-Flach element Euler system, so that its initial class is $\BF_{K_\infty}$) and let $r \in \LL^{\mathfrak{o}}$ be such that $\pmb{\kappa}^{\BF}_{K_\infty}=r\cdot \pmb{\kappa}$. We contend to prove that $r=\varpi^{m}u$ where $m\in \ZZ_{\geq0}$ and $u\in \LL^{\mathfrak{o}}$ is a unit. Indeed, (\ref{eqn:KLZplusminusmainconjintwovar}) together with Theorem~\ref{thm:skinnerurbanwan} (whenever it applies) shows that
\begin{equation}
\label{eqn:twovarequalityintermsofKS}
\Char_{\LL^\mathfrak{o}}\left(\mathfrak{X}_{+-}(f\otimes\alpha/K_{\infty})\right)\dot{=}\, \widetilde{\col}^{(2,\alpha)}\circ \res_{\fp^c}(\BF_{K_\infty})=r\cdot \widetilde{\col}^{(2,\alpha)}\circ \res_{\fp^c}(\kappa_1)
\end{equation}
where ``$\dot{=}$" stands for equality up to powers of $\varpi$ as above. On the other hand, applying the Kolyvagin system machinery on $\pmb{\kappa}$, we deduce that 
\begin{equation}
\label{eqn:twovarequalityintermsofKS2}
\Char_{\LL^\mathfrak{o}}\left(\mathfrak{X}_{+-}(f\otimes\alpha/K_{\infty})\right)\,\,\Big{|}\,\, \widetilde{\col}^{(2,\alpha)}\circ \res_{\fp^c}(\kappa_1)\,.
\end{equation}
up to powers of $\varpi$. Here, we use the fact that the cokernel of $\widetilde{\col}^{(2,\alpha)}$ is pseudo-null up to a power of $\varpi$. Combining (\ref{eqn:twovarequalityintermsofKS}) and (\ref{eqn:twovarequalityintermsofKS2}), we conclude that $r$ has the desired form. 
This shows (using the fact that the $\LL_{\ac}^\mathfrak{o}$-module $\textup{\textbf{KS}}(\FFF_+,\TT_{f,\alpha}^\ac)$ has also rank one and in fact, it is generated by the image $\pmb{\kappa}^{\ac}:=\pi^\ac(\pmb{\kappa})$ of $\pmb{\kappa}$) we have 
$$\pmb{\kappa}^{\BF}_{D_\infty}=\varpi^{m}u_0\cdot \pmb{\kappa}^{\ac}$$
where $u_0=\pi^{\ac}(u)\in \LL_\ac^{\mathfrak{o}}$ is a unit and  $\pmb{\kappa}^{\BF}_{D_\infty}=\pi^\ac(\pmb{\kappa}^{\BF}_{K_\infty})$. The conclusion of Corollary~\ref{cor:mainanticyclosharpdefinite} follows from \cite[Theorem 5.3.10(iii)]{mr02} (whose proof applies verbatim for the Selmer structure $\FFF_+$ used in place of the Selmer structure $\FFF_\LL$ in loc. cit.). 
\end{remark}

We next turn our attention to the indefinite case. We will show in this situation that X. Wan's main result in \cite{xinwanwanrankinselberg} may be utilized (under some technical assumptions) to upgrade the divisibility in (\ref{eqn:KLZplusminusmainconjintwovar}) to an equality up to powers of $\varpi$. Recall from \eqref{eq:WanLfunction} that for each $p$-distinguished character $\alpha$, we have Wan's $p$-adic $L$-function $\calL_{f,\alpha}^{\textup{Hida}}$.
\begin{proposition}
\label{prop:imageunderintegralcolemanmaps}
We have the divisibility
\begin{equation}
\label{eqn:KLZplusminusmainconjintwovar2}
\Char_{\LL^\mathfrak{o}}\left(\mathfrak{X}_{+-}(f\otimes\alpha/K_{\infty})\right)\otimes_{\LL_L} \RR_\Phi\,\,\,\Bigg{|}\,\,\,\left(\calL_{f,\alpha}^{\textup{Hida}}\right)\,.
\end{equation}
in $\RR_\Phi$. The divisibility may be upgraded to an equality if and only if we have equality in (\ref{eqn:KLZplusminusmainconjintwovar}) up to powers of $\varpi$.
\end{proposition}
\begin{proof}
This follows from \eqref{eq:WanLfunction}, (\ref{eqn:KLZplusminusmainconjintwovar}) and the fact that the cokernel of $\widetilde{\col}^{(2,\alpha)}$ is pseudo-null up to a power of $\varpi$.
\end{proof}
\begin{theorem}[Wan]
\label{thm:wanindefinitemain}
Suppose that $N$ is square free as well as that the hypothesis \textup{\textbf{(wt-2)}} holds true. We assume further that there exists a prime $q \mid N^-$ such that $\overline{\rho}_f$ is ramified at $q$. Then the ideal  $\Char_{\LL^\mathfrak{o}}\left(\mathfrak{X}_{+-}(f\otimes\alpha/K_{\infty})\right)\otimes_{\LL_L} \RR_\Phi$ is generated by $\calL_{f,\alpha}^{\textup{Hida}}$.
\end{theorem}

\begin{corollary}
\label{cor:mainconjequalityindefinitecase}
Suppose that $N^-$ is a product of even number of primes and $p\nmid N\rm{disc}(K/\QQ)$ (so that the condition (\textup{Sign }$-$) holds true). Assume in addition that $N$ is square free and that the hypothesis \textup{\textbf{(wt-2)}} is valid. Suppose further that there exists a prime $q \mid N^-$ such that $\overline{\rho}_f$ is ramified at $q$. Then the divisibility in the statement of Theorem~\ref{thm:mainconjwithoutpadicLfunc} may be upgraded to an equality up to powers of $\varpi$.
\end{corollary}
\begin{proof}
Combining Proposition~\ref{prop:imageunderintegralcolemanmaps} and  Theorems \ref{thm:2varmainconjwithoutpadicL}(i) and \ref{thm:wanindefinitemain}, we have equality in (\ref{eqn:KLZplusminusmainconjintwovar}), and therefore also in (\ref{eqn:KLZplusmainconjintwovar}). The proof now follows arguing exactly as in the proof of Corollary~\ref{cor:mainanticyclosharpdefinite}.
\end{proof}
\subsection{$\LL_{\ac}^\mathfrak{o}$-adic height pairing and a Rubin-style formula}
\label{subsec:rubinsformula}
Throughout this section we shall assume (in addition to the hypotheses we have recorded at the start of Section~\ref{subsec:BFlocally}) that $N^-$ is a square-free product of even number of primes (so that the condition (\textup{Sign }$-$) holds true). For every $n$, we let 
$$\mathfrak{h}^{(n)}_p:  H^1_{\FFgr}(D_n,T_{f,\alpha})\times H^1_{\FFgr}(D_n,T_{f,\alpha^{-1}})^\iota \lra \mathfrak{o}$$
denote the Nekov\'a\v{r}'s $p$-adic height pairing over the field $D_n$, defined as in \cite{nekovarheightpairings}. Compatibility of this pairing with restriction and corestriction maps allows us to define its $\LL_{\ac}^\mathfrak{o}$-adic lifting as follows:
\begin{defn}
\label{def:anticycloheight}
We define the $\LL_{\ac}^\mathfrak{o}$-adic height pairing
$$\mathfrak{h}^\ac_p: H^1_{\FFgr}(K,\TT_{f,\alpha}^\ac)\otimes H^1_{\FFgr}(K,\TT_{f,\alpha^{-1}}^{\ac})^\iota \lra \LL_{\ac}^\mathfrak{o}$$
by setting
$$\mathfrak{h}^\ac_p(\mathfrak{x},\mathfrak{y}):=\varprojlim \sum_{\gamma\in \Gamma^\ac_n}\mathfrak{h}^{(n)}_p\left(\pi_n\left(\mathfrak{x}\right),\pi_n\left(\mathfrak{y}\right)^\gamma\right)\cdot\gamma\,\, \in\,\, \LL_{\ac}^\mathfrak{o}$$
where  $\pi_n:  H^1_{\FFgr}(K,\TT_{f,\beta}^\ac) \ra  H^1_{\FFgr}(D_n,T_{f,\beta})$
is the obvious map for $\beta=\alpha,\alpha^{-1}$.
\end{defn}
The $\LL_{\ac}^\mathfrak{o}$-adic regulator $\mathfrak{Reg}^\ac \in \LL_{\ac}^{\mathfrak{o}}$ is defined as the characteristic ideal of the cokernel of $\mathfrak{h}_p^\ac$.
\begin{defn} Define the $\LL_{\ac}^\mathfrak{o}$-adic Tate pairing
$$\langle\,,\,\rangle_{\textup{Tate}}:\,H^1_{+/\f}(K_p,\TT^\ac_{f,\alpha})\otimes H^1_{\FFgr}(K_{\fp^c},\TT_{f,\alpha^{-1}}^\ac)\stackrel{\sim}{\lra} \LL_{\ac}^\mathfrak{o}$$
by setting 
$$\langle\mathfrak{a},\mathfrak{b}\rangle_{\textup{Tate}}:=\varprojlim \sum_{\gamma\in \Gamma^\ac_n}\langle\pi_n(\mathfrak{a})^\gamma,\pi_n(\mathfrak{b})\rangle_{n}\cdot\gamma^{-1}$$
where $\langle\,,\,\rangle_{n}: H^1_{+/\f}(D_{n,\fp},T_{f,\alpha})\otimes H^1_{\FFgr}(D_{n,\fp^c},T_{f,\alpha^{-1}})\ra \mathfrak{o}$ is the usual local Tate pairing. In order to help with our notation here, we would like to remark the natural identification 
$$H^1_{+/\f}(K_p,\TT_{f,\alpha}^\ac)\stackrel{\sim}{\lra} H^1_s(K_{\fp^c},\TT_{f,\alpha}^\ac):=H^1(K_{\fp^c},\TT_{f,\alpha}^\ac)/H^1_{\FFgr}(K_{\fp^c},\TT_{f,\alpha}^\ac)\,.$$
\end{defn}

\begin{theorem}[Rubin-style formula]
\label{thm:rubinstyleformula}Let $\textup{BF}_{K_\infty}\in H^1(K,\TT_{f,\alpha})$ denote the tower of Beilinson-Flach elements along the maximal $\ZZ_p$-power extension $K_\infty$ of $K$.\\
\textup{\bf{(i)}} There exists a unique element 
$$\mathfrak{d}\textup{BF}_{K_\infty} \in H^1_{+/f}(K_p,\TT_{f,\alpha})\otimes\QQ_p$$ 
(the cyclotomic derivative of the Beilinson-Flach element) with the property that
$$\res_{+/\f}\left(\textup{BF}_{K_\infty}\right)=\frac{(\gamma_\cyc-1)}{\log_p\chi_\cyc(\gamma_\cyc)}\cdot\mathfrak{d}\textup{BF}_{K_\infty}\,.$$
\\\textup{\bf{(ii)}} Let $\mathfrak{d}\textup{BF}_{D_\infty}\in H^1_{+/\f}(K,\TT_{f,\alpha}^\ac)$ denote the image of $\mathfrak{d}\textup{BF}_{K_\infty}$. Then,
$$\mathfrak{h}^\ac_p\left(\textup{BF}_{D_\infty},\mathfrak{y}\right)=-\langle \mathfrak{d}\textup{BF}_{D_\infty},\res_{\fp^c}\left(\mathfrak{y}\right)\rangle_{\textup{Tate}}\,.$$
\end{theorem}
\begin{proof}
The first portion holds true thanks to our assumption that $\epsilon(f/K)=-1$ and Theorem~\ref{thm:explicitreciprocity}. The second part is a formal consequence of Nekov\'a\v{r}'s general result \cite[Proposition 11.3.15]{nekovar06}  applied along the anticyclotomic Iwasawa tower (where the vanishing of the terms in Nekov\'a\v{r}'s formula at primes away from $p$ follow from the local Langlands correspondence). We further note that the construction of the $p$-adic height pairing in loc. cit. compares to its more classical versions in \cite{nekovarheightpairings} via \cite[\S 11.3]{nekovar06}.
 \end{proof}
 
 \begin{proposition}
 \label{prop:thederivativemapstoderivative}
 The restriction of the Coleman-Perrin-Riou map $\col^{(1,\alpha)}$ (that we introduced in Section~\ref{subsec:BFlocally}) to the anticyclotomic tower sends $\mathfrak{d}\textup{BF}_{D_\infty}$ to $\al_{f,1}^{(\alpha)}$.
 \end{proposition}
 \begin{proof}
 This follows from definitions and the explicit reciprocity law.
 \end{proof}
\begin{theorem}
\label{thm:torsionpartGreenberg} If we assume that $p\nmid N\rm{disc}(K/\QQ)$, then we have a containment
$$\al_{f,1}^{(\alpha)}\in \mathfrak{Reg}_{\ac}\cdot\Char\left(\mathfrak{X}(f\otimes\alpha/D_\infty)_{\textup{tor}}\right)\otimes_{\LL_\ac^\mathfrak{o}} \RR_L^\ac\,.$$ If we assume in addition that
\begin{itemize}
\item $N$ is square free;
\item the hypothesis \textup{\textbf{(wt-2)}} holds true;
\item there exists a prime $q \mid N^-$ such that $\overline{\rho}_f$ is ramified at $q$,
\end{itemize}
then the ideal $\mathfrak{Reg}_{\ac}\cdot\Char\left(\mathfrak{X}(f\otimes\alpha/D_\infty)_{\textup{tor}}\right)\subset \RR_L^\ac$ is in fact generated by $\al_{f,1}^{(\alpha)}\,.$
\end{theorem}
\begin{proof}
It follows from Theorem~\ref{thm:rubinstyleformula}, Proposition~\ref{prop:thederivativemapstoderivative} (and the fact that Tate pairing is an isomorphism) that 
\begin{align}
\notag \mathfrak{L}_{f,1}^{(\alpha)}\cdot \mathfrak{l}_-&=\mathfrak{Reg}_{\ac}(\textup{BF}_{D_\infty})\\
\label{eqn:firstreductionmain} &=\mathfrak{Reg}_{\ac}\cdot \Char\left(H^1_{\FFgr}(K,\TT_{f,\alpha}^\ac)\big{/}\LL_{\ac}^\mathfrak{o}\cdot\textup{BF}_{D_\infty}\right)
\end{align}
where $\mathfrak{Reg}_{\ac}(\textup{BF}_{D_\infty})$ is the characteristic ideal of the cokernel of the map
$$\mathfrak{h}_p^\ac: \left(\LL_{\ac}^\mathfrak{o} \cdot \textup{BF}_{D_\infty} \right)\otimes H^1_{\FFgr}(K,\TT_{f,\alpha^{-1}}^\ac)\lra \LL_{\ac}^\mathfrak{o}\,.$$
Combining (\ref{eqn:firstreductionmain}) and Theorem~\ref{thm:mainconjwithoutpadicLfunc}, we conclude that
\begin{equation}
\label{eqn:secondreductionmain}
\mathfrak{L}_{f,1}^{(\alpha)}\cdot \mathfrak{l}_-\subseteq\mathfrak{Reg}_{\ac}\cdot\Char\left(\mathfrak{X}_+(f\otimes\alpha/D_\infty)\right)
\end{equation}
as ideals of $\RR_L^\ac$. Theorem~\ref{prop:torsionsubmodulepofXminus} together with (\ref{eqn:secondreductionmain}) shows that
$$\mathfrak{L}_{f,1}^{(\alpha)}\cdot \mathfrak{l}_-\subseteq\mathfrak{Reg}_{\ac}\cdot\Char\left(\mathfrak{X}_-(f\otimes\alpha/D_\infty)_{\textup{tor}}\right)
$$
and this combined with Lemma~\ref{lem:PTdualityandtorsionsubmodules} concludes the proof that 
$$\mathfrak{L}_{f,1}^{(\alpha)}\cdot \mathfrak{l}_-\subseteq\mathfrak{Reg}_{\ac}\cdot\Char\left(\mathfrak{X}(f\otimes\alpha/D_\infty)_{\textup{tor}}\right)\cdot \mathfrak{l}_-\,.$$
By rank considerations, the ideal $\mathfrak{l}_-$ is non-zero and the first part of the theorem is proved. Under the additional hypotheses, Corollary~\ref{cor:mainconjequalityindefinitecase} applies and allows us to upgrade all the containments above to equalities and therefore completes the proof of the theorem.
\end{proof}
\bibliographystyle{amsalpha}
\bibliography{references}
\end{document}